\documentclass[11pt, oneside,reqno]{amsart}  	

\usepackage{geometry}               
\usepackage{graphicx}				
\usepackage{amsthm,amsmath,amssymb,graphicx,mathrsfs,braket,latexsym}
\usepackage[all]{xy}
\usepackage[monochrome]{xcolor}
\usepackage{ascmac}
\usepackage{subfigure}
\usepackage{url}

\DeclareFontFamily{U}{mathx}{\hyphenchar\font45}
\DeclareFontShape{U}{mathx}{m}{n}{ <5> <6> <7> <8> <9> <10>
   <10.95> <12> <14.4> <17.28> <20.74> <24.88> mathx10 }{}
\DeclareSymbolFont{mathx}{U}{mathx}{m}{n}
\DeclareFontSubstitution{U}{mathx}{m}{n}
\DeclareMathAccent{\widecheck}{0}{mathx}{"71}

\newtheorem{thm}{Theorem}[section]
\newtheorem{prop}[thm]{Proposition}
\newtheorem{lem}[thm]{Lemma}

\theoremstyle{definition}
\newtheorem{defi}[thm]{Definition}
\newtheorem{rem}[thm]{Remark}
\newtheorem{ex}[thm]{Example}

\newtheorem{setting}[thm]{Setting}

\DeclareMathOperator{\supp}{supp}

\DeclareMathOperator{\dens}{dens}
\DeclareMathOperator{\freq}{freq}
\DeclareMathOperator{\Tr}{Tr}

\newcommand{\card}{\#}

\newcommand{\Zpo}{\mathbb{Z}_{>0}}

\newcommand{\calT}{\mathcal{T}}

\newcommand{\calP}{\mathcal{P}}
\newcommand{\calQ}{\mathcal{Q}}

\newcommand{\e}{\varepsilon}
\newcommand{\im}{\mathrm{i}}

\newcommand{\R}{\mathbb{R}}
\newcommand{\C}{\mathbb{C}}
\newcommand{\T}{\mathbb{T}}

\def\maprestriction#1#2{\left. #1 \right|_{#2}}

\title[Absence of absolutely continuous diffraction for S-adic tilings]{Absence of absolutely continuous diffraction spectrum for certain S-adic tilings}
\author{Yasushi Nagai}
\address{School of General Education,  Shinshu University, 3-1-1, Asahi, Matsumoto, Nagano, 
390-8621, Japan}
\email{ynagai@shinshu-u.ac.jp}
\date{\today}							

\begin{document}
\maketitle

\begin{abstract}
      \textcolor{red}{Quasiperiodic tilings} are often considered as structure models of quasicrystals.
      In this context, it is important to study the nature of the diffraction measures for tilings.
      In this article, we investigate the diffraction measures for S-adic tilings
     \textcolor{red}{in $\R^d$}, which are
      constructed from a family of geometric substitution rules.
      In particular, we firstly give a sufficient condition for the absolutely continuous component
      of the diffraction measure for an S-adic tiling to be zero.
      Next, we prove this sufficient condition for ``almost all'' binary \textcolor{red}{block-substitution} cases
      and thus prove the absence of the absolutely continuous diffraction
      spectrum for most of S-adic tilings from a family of binary \textcolor{red}{block} substitutions.
\end{abstract}

\section{Introduction}

A tiling is a cover of $\R^{d}$ by its countably many subsets (tiles)
$T$ with the property that $T=\overline{T^{\circ}}$ (i.e., each tile is
the closure of its interior).  There exist tilings $\mathcal{T}$ that
are non-periodic (meaning that $\mathcal{T}+x=\mathcal{T}$ holds for
$x=0$ only) but still admit repetitions of patterns: for example,
$\mathcal{T}$ may be repetitive \cite[Definition
  5.8]{Baake-Grimm_vol1}, or almost periodic in a sense, such as in
\cite[Chapter 5]{Baake-Grimm-vol2} and in \cite{G}, to name a few.
For this reason, such tilings are often considered as structure models
of quasicrystals.  The diffraction measures defined for these tilings
then correspond to physical diffraction patterns.  In this context, it
is important to study the nature of diffraction measures for tilings.
Especially, it is interesting to know when a diffraction measures is
pure point (a sum of point or Dirac measures).

There are several ways to construct interesting non-periodic
tilings. One of the most common approaches is via substitution (or
inflation) rules. (There are ``symbolic'' substitution rules and
``geometric'' ones, the spectrum of which are related
\cite{Clark-Sadun}, but in this article we only deal with
``geometric'' ones.)  Given a substitution rule $\rho$ \textcolor{red}{in $\R^d$}, it
gives rise to \textcolor{red}{self-affine} tilings, which are often repetitive and
almost-periodic. The class of self-affine tilings is included in the
class of S-adic tilings, which are tilings that are generated by a finite
family of substitution rules.

Concerning the spectral properties of self-affine tilings, a key
conjecture is the Pisot substitution conjecture, which states that
self-affine tilings obtained from substitution rules of Pisot type
are pure point diffractive, that is, their diffraction measures are
pure point.  This is still an open problem, but there are several
partial positive answers.  Here, we just mention that the binary one-dimensional case,
in which there are only two tiles up to translation, is solved
\cite{Sirvent-Solomyak}. The definition of Pisot type for substitution
includes irreducibility, but for some reducible cases, in the setting
where the substitution is binary \textcolor{red}{block-substitution}, Ma\~{n}ibo
\cite{Manibo_binary, Manibo_thesis} and \textcolor{red}{Baake-Grimm \cite{Baake-Grimm_renormalisation}}
proved the absence of absolutely continuous
components in the diffraction pattern.

In this article, we study the diffraction spectrum for
S-adic tilings in $\R^d$, which generalizes the single substitution (self-affine) 
case \cite{Baake-Grimm_renormalisation, Manibo_binary, Manibo_thesis}.
 In particular, we generalize the method from
\cite{BGM,Manibo_thesis} to prove that (I) an inequality for Fourier
matrices is sufficient for the absence of an absolutely continuous
component in the diffraction measure, for quite a general class of
S-adic tilings (including, but not only, the binary case), and (II)
the sufficient condition in (I) is satisfied for ``almost all'' binary
\textcolor{red}{block-substitution} cases, and so, for such an S-adic tiling, the
absolutely continuous part of the diffraction measure is zero.  
The
precise statement for claim (I) is found in
Theorem~\ref{thm_liminf_implies_zero_acpart}, in the setting specified
in Setting~\ref{setting_relation_exponents_absconti-spec}. 
 The special case for claim (II) is elaborated on below, and the
precise statement for claim (II) is
Theorem~\ref{thm_absence_for_binary}, where the setting for this
result is detailed in Setting~\ref{setting_binary_const_length}.
The key ingredients are renormalization technique developed by
\cite{BFGR, Baake-Gahler, Baake-Grimm_renormalisation, 
 Baake-Grimm-Manibo, Manibo_binary, Manibo_thesis}
 and Furstenberg--Kesten and Oseledets theorems.

\begin{figure}[h]
\begin{center}
\subfigure[an example of block substitution]{
\includegraphics[width=6cm]{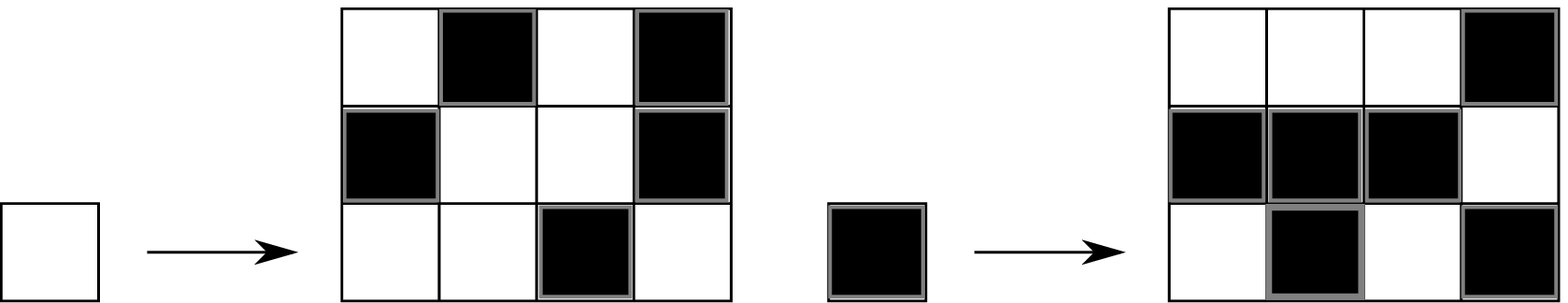}
\label{fig_block_substi1}
}
\hspace{1cm}
\subfigure[another example of block substitution]{
\includegraphics[width=6cm]{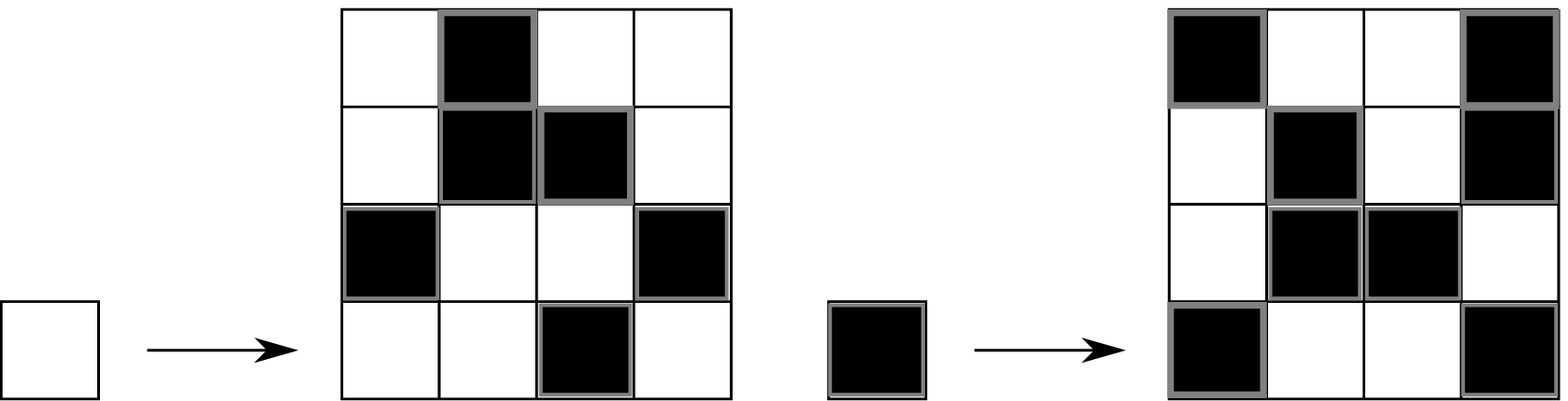}
\label{fig_block_substi2}
}
\caption{Examples of block substitutions}
\label{Fig1}
\end{center}
\end{figure}

\textcolor{red}{To elaborate on the claim (II), let us consider two substitutions, $\rho_1$ in Figure \ref{fig_block_substi1}
and $\rho_2$ in Figure \ref{fig_block_substi2}. 
Such substitutions (one with prototiles with support $[0,1]^d$) are called block substitutions.
For arbitrary sequence $i_1,i_2,\ldots$ in $\{1,2\}^{\mathbb{N}}$,
by choosing an appropriate increasing sequence $n_1<n_2<\cdots$ and appropriate patches $\mathcal{P}_k$, $k=1,2,\ldots$, 
we have a convergence}
\begin{align*}
        \mathcal{T}=\lim_{k\rightarrow\infty}\rho_{i_1}\circ\rho_{i_2}\circ\cdots\circ\rho_{i_{n_k}}(\mathcal{P}_k)
\end{align*}
\textcolor{red}{and $\mathcal{T}$ is a tiling. (For details, see page \pageref{eq_def_S-adic-tiling}.)
Such $\mathcal{T}$ is called an S-adic tiling belonging to the sequence $(i_n)_n$
for $\rho_1,\rho_2$.
The special case for the main result of this paper (Theorem \ref{thm_absence_for_binary}) is as follows.}

\begin{thm}[A special case of Theorem \ref{thm_absence_for_binary}]\label{thm_special_case1}
       \textcolor{red}{Let $p_1,p_2$ be two positive real numbers with $p_1+p_2=1$. Endow $\{1,2\}^{\mathbb{N}}$ the product probability measure $\mu$
       for the probability measure on $\{1,2\}$ defined by $(p_1,p_2)$. Then, for $\mu$-almost all $(i_n)_n\in\{1,2\}^{\mathbb{N}}$,
       the S-adic tilings belonging to $(i_n)_n$ for $\rho_1,\rho_2$ have zero absolutely continuous diffraction spectrum,
       that is, the absolutely continuous part of the diffraction measure is zero.}
\end{thm}

\textcolor{red}{We can replace $\{1,2\}^{\mathbb{N}}$ with its subshift $X$, as follows:}

\begin{thm}[A special case of Theorem \ref{thm_absence_for_binary}]\label{thm_special_case2}
       \textcolor{red}{Let $X$ be a subshift of $\{1,2\}^{\mathbb{N}}$ which admits shift-invariant ergodic Borel probability measure $\mu_X$.
       Assume the shift map on $X$ is surjective. Then, for $\mu_X$-almost all $(i_n)_n\in X$,
       the S-adic tilings belonging to $(i_n)_n$ for $\rho_1,\rho_2$ have zero absolutely continuous diffraction spectrum.}
\end{thm}
\textcolor{red}{Note that this is not included in Theorem \ref{thm_special_case1} because $\mu(X)$ might be zero.}

\textcolor{red}{We can replace $\rho_1,\rho_2$ with arbitrary finite family of binary block substitutions, with a mild assumption on
substitution matrices,  with possibly different
expansion maps. Furthermore, the dimension can be arbitrary:
for any $d=1,2,3\ldots$ and block substitutions in $\R^d$, we have similar results.}

This paper is organized as follows. In
Section~\ref{section_general_background}, we introduce our notation
and some necessary background.  Section~\ref{section_main-result}
contains our main results; in particular, we state and prove claims
(I) and (II) given above.  The first claim is proved in
Section~\ref{subsection_sufficient-condition}, while the second
claim is proved in Section~\ref{subsection_absence_acpart}.  We
defer the proofs of some of our claims in
Section~\ref{subsection_absence_acpart} to an appendix.

\section{General background}
\label{section_general_background}
\subsection{Notations}
For a finite set $F$, we denote its cardinality by $\card F$.  \textcolor{red}{In this
article, $\mu_{\mathrm{L}}$ denotes the Lebesgue measure on $\R^d$.}
The symbol $\mathbb{T}$ refers to the one-dimensional torus
$\{z\in\mathbb{C}\mid |z|=1\}$.  For a natural number $n$, we will identify $\T^n$
measure-theoretically with $[0,1)^n$, on which the Lebesgue
  measure $\mu_{\mathrm{L}}$ is the complete rotation-invariant probability
  measure.  Let $\pi\colon\R^n\rightarrow\mathbb{T}^n$ be defined via
  $\pi(s_1,s_2,\ldots ,s_n)=(e^{2\pi\im s_1},e^{2\pi\im s_2},\ldots ,e^{2\pi\im s_n})$.
  In $\R^d$, for $x\in\R^d$ and $R>0$, the closed ball
  $\{y\in\R^d\mid \|x-y\|\leqq R\}$ is denoted by $B(x,R)$. If $x=0$, we use the symbol
  $B_R$ for $B(0,R)$.


\subsection{A generality for tilings and substitutions}
  
In this section, we sketch a generality for the theory of tilings.  For
a detailed exposition, we refer to \cite{Baake-Grimm_vol1}.  
Let $d$ be a natural number and we consider tilings in $\mathbb{R}^d$.

Let $L$ be a finite set.  A \emph{\textcolor{red}{labelled} tile} is a pair $T=(S,\ell)$
consisting of a compact set $S$ in $\R^d$ with $\overline{S^{\circ}}=S$ 
(the closure of the interior coincides with the original $S$) and an element
$\ell\in L$. The set $S$ is called the \emph{support} of $T$ and
denoted by $\supp T$. The element $\ell$ is called the \emph{label} of
$T$.

Alternatively, we can consider ``unlabeled'' tiles, that is, a compact subset $T$ of $\R^d$ such that $\overline{T^{\circ}}=T$ .  
For an unlabeled tile $T$, we denote the
space it covers (that is, $T$ itself) by $\supp T$, called the support
of $T$, in order to cover the theory for labeled and for unlabeled
tiles by the same notation.  
\textcolor{red}{Both labelled tiles and unlabelled tiles are called tiles.}
We deal with both cases simultaneously by the above
notation. Note that each of the cases are required because (1) we often have
to distinguish two tiles with the same support by assigning them
different labels, as in Example~\ref{ex-Thue-Morse}, and because (2)
we often meet situations where tiles have different support and labels
are hence redundant, as in Example \ref{ex-Fibonacci}.

A set $\mathcal{P}$ of tiles in $\R^d$ is called a \emph{patch} if $(\supp
T)^{\circ}\cap (\supp S)^{\circ}=\emptyset$ for each distinct $S$ and
$T$ in $\mathcal{P}$.  The support of a patch $\mathcal{P}$ is the subset
$\bigcup_{T\in\mathcal{P}}\supp T$ of $\R$ and is denoted by
$\supp\mathcal{P}$. (Sometimes we take the closure after taking the union in
this definition, but in this article we only deal with situations
where the union is already a closed set.  We use the same notation as
the support of a tile, but there is no possibility of confusion.)  A
patch $\mathcal{P}$ is called a tiling if $\supp\mathcal{P}=\R^d$.

\textcolor{red}{For an unlabelled tile $S$, $S+x$ denotes the usual translation.}
For a labelled tile $T=(S,\ell)$ and $x\in\R^d$, we set $T+x=(S+x,\ell)$. 
For a
patch $\mathcal{P}$ \textcolor{red}{(with either labelled or unlabelled tiles)} and $x\in\R^d$, we define the translate of $\mathcal{P}$
by $x$ via
\begin{align*}
  \mathcal{P}+x=\bigl\{T+x\mid T\in\mathcal{P}\bigr\}.
\end{align*}
A tiling $\mathcal{T}$ is said to be \emph{non-periodic} if $x=0$ is
the only element in $\R^d$ that satisfies
$\mathcal{T}+x=\mathcal{T}$. In this article, we are mainly interested
in non-periodic tilings.

There are several ways to construct interesting non-periodic tilings.
In this article, we consider tilings constructed via
substitution rules.  
First, for a finite set $\mathcal{A}$ of tiles in $\R^d$, let $\mathcal{A}^{*}$ be
the set of all patches of which tiles are translates of elements of
$\mathcal{A}$.  A \emph{substitution rule} (or an \emph{inflation rule}) is
a triple $\sigma=(\mathcal{A},\phi,\rho)$ where
\begin{itemize}
\item $\mathcal{A}$ is a finite set of tiles, called the \emph{alphabet} of $\sigma$,
\item  $\phi\colon\R^d\rightarrow\R^d$ is a linear map \textcolor{red}{with $\min_{\|v\|=1}\|\phi(v)\|>1$},
called the \emph{expansion map}, and
\item $\rho$ is a map $\mathcal{A}\rightarrow\mathcal{A}^{*}$ such that
\begin{align*}
\supp\rho(T)=\phi(\supp T)
\end{align*}
holds for each $T\in\mathcal{A}$.
\end{itemize}
 The map $\rho$ itself is also often referred to
as a substitution (or inflation) rule. 
\textcolor{red}{Usually, the expansion map is defined as a linear map whose eigenvalues are greater than $1$ in modulus, but
for a technical reason, we use a stronger definition.}
The third condition (on the
supports) means that the map $\rho$ gives the result of first
expanding the tile $T$ by the expansion map $\phi$ and then
subdividing it to obtain a patch $\rho(T)$.  The following examples
will illustrate this point.

\begin{ex}\label{ex-Thue-Morse}
Let us consider the case where $d=1$.
Let $T_{1}=([0,1],1)$ and $T_{2}=([0,1],2)$. The \emph{Thue--Morse
  substitution} is a substitution $\rho^{}_{\text{TM}}$ of which alphabet is
$\{T_{1},T_{2}\}$, expansion map is $\R\ni x\mapsto 2x\in\R$ and the rule is given by
\begin{align*}
\rho^{}_{\text{TM}}(T_{1})&=\{T_{1},T_{2}+1\}\\
\rho^{}_{\text{TM}}(T_{2})&=\{T_{2},T_{1}+1\}.
\end{align*}
The final condition in the definition of a substitution rule is indeed
satisfied for this rule, since $\supp\rho^{}_{\text{TM}}(T_{i})=[0,2]$
and $2\supp T_{i}=[0,2]$ for $i=1,2$.
\end{ex}

\begin{ex}\label{ex-Fibonacci}
Again, consider the case where $d=1$.
Set $\tau=\frac{1+\sqrt{5}}{2}$, the golden ratio.
Let $\mathcal{A}=\{T_{a},T_{b}\}$, where $T_{a}=[0,\tau]$ and $T_{b}=[0,1]$.
The \emph{Fibonacci substitution} is the map $\rho^{}_{\text{F}}$
\begin{align*}
\rho^{}_{\text{F}}(T_{a})&=\{T_{a}, T_{b}+\tau\},\\
\rho^{}_{\text{F}}(T_{b})&=\{T_{a}\}.
\end{align*}
Again, with an expansion map $\R\ni x\mapsto \tau x\in\R$,
the final condition in the definition of a substitution rule is
satisfied since $\tau^{2}=\tau+1$.
\end{ex}

\begin{ex}\label{ex_block}
\textcolor{red}{
       A substitution rule such that the supports of elements of the alphabet are all $[0,1]^d$ and the
       expansion map is a diagonal matrix with natural numbers greater than 1 as diagonal entries is
       called a block substitution. For example, Figure \ref{fig_block_substi1} and \ref{fig_block_substi2}
       are block substitutions with a common alphabet $\mathcal{A}=\{([0,1]^2,B), ([0,1]^2,W)\}$ and an expansion map defined by}
       \begin{align*}
                   \begin{pmatrix}
                        4,&0\\
                        0,&3
                   \end{pmatrix},
       \end{align*}
       \textcolor{red}{and one by $4I$ ($I$ being the identity matrix), respectively.}
\end{ex}

For a substitution rule $\rho$, one can define a displacement matrix
and Fourier matrix, which will play important roles in the study of
diffraction, as follows.  Let
$\mathcal{A}=\{T_{1},T_{2},\ldots,T_{n_{a}}\}$ be the alphabet for the
substitution $\rho$ in $\R^d$.  For each $i$ and $j$, there is a \emph{digit
  set} $T_{i,j}\subset\R^d$ for $\rho$, which is determined by
\begin{align}
       \rho(T_{j})=\bigl\{T_{i}+x\mid i\in\{1,2,\ldots ,n_{a}\}, x\in T_{i,j}\bigr\}.
\end{align}
The \emph{substitution matrix} of $\rho$ is the matrix whose
$(i,j)$-element is $\card T_{i,j}$. 
We then define the \emph{Fourier matrix}, $B$, which is a
$n_{a}\times n_{a}$ matrix function on $\R^d$.  We need to specify its
value $B(t)$ for each $t\in\R^d$. We define the $(i,j)$ component of
$B(t)$, denoted by $B_{i,j}(t)$, as
\begin{align*}
      B_{i,j}(t) = \sum_{s\in T_{i,j}}e^{2\pi\im\langle s,t\rangle},
\end{align*}
where $\langle\cdot,\cdot\rangle$ is the standard inner product in $\R^d$.
Let us consider an explicit example to illustrate these definitions.

\begin{ex}
    For \textcolor{red}{the} Thue--Morse substitution $\rho^{}_{\text{TM}}$ from Example
    \ref{ex-Thue-Morse},
    the Fourier matrix is
    \begin{align*}
          B(t)=
          \begin{pmatrix}
                1, &e^{2\pi\im t}\\
                e^{2\pi\im t}, &1
          \end{pmatrix}.
    \end{align*}
\end{ex}

Given a geometric substitution rule $\rho$ in $\R^d$, one can construct a tiling in $\R^d$
by iterating the map $\rho$.  To be more precise, for a given geometric
substitution rule $\rho$, we can define a map
$\rho\colon\mathcal{A}^{*}\rightarrow\mathcal{A}^{*}$ (denoted by the
same symbol), as follows.  First, for $T\in\mathcal{A}$ and $x\in\R^d$,
set $\rho(T+x)=\rho(T)+\phi(x)$ ($\phi$ being the expansion
map). Then we define $\rho(\mathcal{P})$, where $\mathcal{P}$ is a
patch consisting of translates of elements of $\mathcal{A}$ (that is,
an element of $\mathcal{A}^{*}$), via
\begin{align*}
  \rho(\mathcal{P})=\bigcup_{T\in\mathcal{P}}\rho(T).
\end{align*}
Since now the domain and the range of the new map $\rho$ are the same, we can iterate it.
We can often take the limit
\begin{align}
    \lim_{n\rightarrow\infty}\rho^{kn}(\mathcal{P})\label{limit-def-self-similar-tiling}
\end{align}
to obtain a tiling, for a suitable $k>0$ and an initial patch
$\mathcal{P}$.  The convergence in equation
\eqref{limit-def-self-similar-tiling} is with respect to the
\emph{local matching topology}, in which two patches $\mathcal{P}$ and
$\calQ$ are ``close'' if there are small displacements $x,y\in\R^d$ such
that $\mathcal{P}+x$ and $\calQ+y$ agree inside $B_R$ for some
large $R>0$.  (See, for example, \cite[p.129]{Baake-Grimm_vol1}.)

\begin{ex}
For the Thue--Morse substitution $\rho^{}_{\text{TM}}$, define $\mathcal{P}$ via
\begin{align*}
       \mathcal{P}=\{T_{1}-1, T_{0}\}.
\end{align*}
Then, $\rho_{\text{TM}}^{2}(\mathcal{P})$ is (if we write it
symbolically) $1001.0110$, where $.$ denotes the place of origin. We
obtain $\rho_{\text{TM}}^{2}(\mathcal{P})\supset\mathcal{P}$, and this
in turn means that
$\rho_{\text{TM}}^{2}(\mathcal{P})\subset\rho_{\text{TM}}^{4}(\mathcal{P})\subset\rho_{\text{TM}}^{6}(\mathcal{P})\cdots$.
The patches obtained by iteration ``grow'' in $\R$, and in the limit they form a tiling
\begin{align*}
  \lim_{n\rightarrow\infty}\rho_{\text{TM}}^{2n}(\mathcal{P})
  =\bigcup_{n>0}\rho_{\text{TM}}^{2n}(\mathcal{P}),
\end{align*}
which is called a \emph{Thue--Morse tiling}.
\end{ex}

An \emph{S-adic tiling} is a tiling obtained by replacing each of $kn$
$\rho$'s in $\rho^{kn}(\mathcal{P})$ in equation
\eqref{limit-def-self-similar-tiling} with a substitution rule from a
finite set of geometric substitution rules. To be precise,
consider a finite set $\{\rho_{1},\rho_{2},\ldots,\rho_{m_{a}}\}$ of
geometric substitution rules in $\R^d$ that share the same alphabet
$\mathcal{A}$ but do not necessarily share the same expansion map.
We call sequences $i_{1},i_{2},\ldots$ of elements
of $\{1,2,\ldots, m_{a}\}$ \emph{directive sequences}. Given a
directive sequence $i_{1},i_{2},\ldots$, any tiling $\mathcal{T}$ of
the form
\begin{align}
  \mathcal{T} = \lim_{l\rightarrow\infty}\rho_{i_{1}}\circ\rho_{i_{2}}\circ
  \cdots\textcolor{red}{\circ}\rho_{i_{n_{l}}}(\mathcal{P}_{l}),
\label{eq_def_S-adic-tiling}
\end{align}
where $n_{1}<n_{2}<\cdots$ and where \textcolor{red}{the $\mathcal{P}_{l}$ are patches
that are included in some $\rho_{j_1}\circ\rho_{j_2}\circ\cdots\circ\rho_{j_m}(P)$($m>0, j_1,j_2,\ldots, j_m\in\{1,2,\ldots, m_a\}$ and
$P\in\mathcal{A}$),
 is called an \emph{S-adic
  tiling belonging to the directive sequence $i_{1},i_{2},\cdots$ for the family $\{\rho_1,\rho_2,\ldots, \rho_{m_a}\}$}.
  This is a geometric version of the symbolic S-adic sequences (see for example \cite{Berthe-Delecroix}) and 
  the order of $\rho_{i_j}$ in \eqref{eq_def_S-adic-tiling} comes from the symbolic counterpart.
  The convergence in \eqref{eq_def_S-adic-tiling} is assured by the following finiteness condition.
  In general, the family $\{\rho_1,\rho_2,\ldots,\rho_{m_a}\}$ is said to have
  \emph{finite local complexity (FLC)} \label{def_FLC} if for each compact $K\subset\R^d$ the set
  \begin{align*}
        \{\rho_{j_1}\circ\rho_{j_2}\circ\cdots\circ\rho_{j_n}(P)\sqcap (K+x)
       \mid P\in\mathcal{A}, n>0, j_1, j_2,\ldots ,j_n\in\{1,2,\ldots ,m_a\}, x\in\R^d\}
  \end{align*}
  is finite up to translation, where the symbol $\sqcap$ is defined via
  \begin{align*}
        \mathcal{P}\sqcap S=\{T\in\mathcal{P}\mid \supp T\cap S\neq\emptyset\}
  \end{align*}
  for a patch $\calP$ in $\R^d$ and an $S\subset\R^d$.
If the set of substitutions $\{\rho_1,\rho_2,\ldots,\rho_{m_a}\}$ have FLC,
given an arbitrary directive sequence $i_{1},i_{2},\ldots$,
we can find some $n_{1}<n_{2}<\cdots$ and some patches $\mathcal{P}_{l}$
such that the limit in \eqref{eq_def_S-adic-tiling}
converges, because the patches after the $\lim$ symbol in \eqref{eq_def_S-adic-tiling}
are included in a compact set.
 This is seen by the fact that a space $X$ of patches in $\R^d$ such that for each compact $K\subset\R^d$
 \begin{align*}
       \{\calP\sqcap (K+x)\mid \calP\in X, x\in\R^d\}
 \end{align*}
 is finite up to translation is relatively compact, }
 by the standard diagonalization argument (\cite[Theorem
  1.1]{Sadun} or \cite[Corollary 3.20 and Lemma
  3.24]{Nagai_local-matching-top}).
  \textcolor{red}{We can start with a sequence $(\calQ_n)_n$ of patches and
  the sequence}
  \begin{align*}
          \rho_{i_1}\circ\rho_{i_2}\circ\cdots\circ\rho_{i_n}(\calQ_n), n=1,2,\ldots
  \end{align*}
  \textcolor{red}{admits a convergent subsequence.}

In the discussion of \emph{symbolic} S-adic sequences, we can consider
cases where the symbolic substitution rules do not share the same
alphabet, but in this article we only deal with the case where
substitutions are geometric and share a common alphabet.  
\textcolor{red}{This is a strong assumption but all the block substitutions,
which we mainly deal with in this paper, are included in our scope.}
Often, given
a directive sequence $i_{1},i_{2},\ldots$, we use the notation
\begin{align}
       \rho_{i[k,l)}=\rho_{i_{k}}\circ\rho_{i_{k+1}}\circ\textcolor{red}{\cdots}\circ\rho_{i_{l-1}},
\end{align}
for two positive integers $k<l$.

Given an S-adic tiling of the form \eqref{eq_def_S-adic-tiling}, the
sequence $\bigl(\rho_{i[2,n_{l})}(\mathcal{P}_{l})\bigr)_{l>0}$ admits a
  convergent subsequence, again by a diagonalization argument as
  above. We can take a subsequence
  $\bigl(n^{(2)}_{l},\mathcal{P}^{(2)}_{l}\bigr)_{l}$ of the sequence
  $(n_{l},\mathcal{P}_{l})$ so that that the limit
  $\lim_{l}\rho_{[2,n^{(2)}_{l})}(\mathcal{P}^{(2)}_{l})$
    converges. We can further take a subsequence
    $\bigl(n^{(3)}_{l},\mathcal{P}^{(3)}_{l}\bigr)_{l}$ of
    $\bigl(n^{(2)}_{l},\mathcal{P}^{(2)}_{l}\bigr)_{l}$ such that the limit
    $\lim_{l}\rho_{i[3,n^{(3)}_{l})}(\mathcal{P}^{(3)}_{l})$
      converges.  Proceeding in this way, we can take nested
      subsequences $\bigl(n^{(k)}_{l}, \mathcal{P}^{(k)}_{l}\bigr)_{l}$ for
      $k=1,2,3,\ldots$. We set $m_{l}=n^{(l)}_{l}$ and
      $\mathcal{Q}_{l}=\mathcal{P}^{(l)}_{l}$ for $l=1,2,\ldots$. Then, we
      have convergences
\begin{align*}
      \mathcal{T}^{(k)}=\lim_{l\rightarrow\infty}\rho_{i[k,m_{l})}(\mathcal{Q}_{l})
\end{align*}
for each $k>0$ with common $(m_{l})^{}_{l}$ and
$(\mathcal{Q}_{l})^{}_{l}$. This implies that, for each $k$, we have
$\rho_{k}(\mathcal{T}^{(k+1)})=\mathcal{T}^{(k)}$. These
``de-substituted tilings''\label{explanation_de-substitution}
$\mathcal{T}^{(2)},\mathcal{T}^{(3)},\ldots$ of the given
$\mathcal{T}=\mathcal{T}^{(1)}$ will be useful
later;
\textcolor{red}{such an inverse-limit structure enables us to construct renormalization scheme,
by which we can use ergodic theory to study the diffraction spectrum for $\mathcal{T}^{(1)}$.}

\subsection{Patch frequencies}
In order to discuss the diffraction of tilings, we use the concept of
the frequency of patches.  In general, if $\mathcal{T}$ is a tiling in $\R^d$, if
$\mathcal{P}$ is a (usually finite) non-empty patch and if the limit
\begin{align*}
  \lim_{R\rightarrow\infty}\frac{1}{\mu_{\mathrm{L}}(B_R)}
  \card\bigl\{t\in B_R\mid\mathcal{P}+t\subset\mathcal{T}\bigr\}
\end{align*}
converges, this limit is called the \emph{frequency of $\mathcal{P}$
  in $\mathcal{T}$} and denoted by $\freq_{\mathcal{T}}\mathcal{P}$ or $\freq\mathcal{P}$.  (Here, we
consider averaging with respect to $\bigl\{B_R\mid R>0\bigr\}$, but we
can also consider averaging along  van Hove sequences.)  If
$\mathcal{T}$ is an S-adic tiling, often the following \emph{uniform
  patch frequency} holds.

\begin{thm}\label{thm_S-adic_patch-freq}
         Let\/ $\rho_{1},\rho_{2},\ldots,\rho_{m_{a}}$ be (geometric)
         substitution rules in $\R^d$ that share a common alphabet.
          Let\/ $A_{i}$ be the substitution matrix
         for\/ $\rho_{i}$.  Take a directive sequence\/
         $i_{1},i_{2},\ldots\in\{1,2,\ldots, m_{a}\}$ and an S-adic
         tiling\/ $\mathcal{T}$ belonging to this directive sequence.
         Assume the following four conditions:
         \begin{enumerate}
         \item there are\/ $n_{0}>0$ and\/
           $i_{0,1},i_{0,2},\ldots, i_{0,n_{0}}\in\{1,2,\ldots, m_{a}\}$ such that
         all entries in the product matrix
         \begin{align*}
                A_{i_{0,1}}A_{i_{0,2}}\cdots A_{i_{0,n_{0}}}
         \end{align*}
         are greater than\/ $0$;
         \item for any\/ $n>0$ there is\/ $k>n$ such that
         \begin{align*}
                    i_{k}=i_{0,1},\quad i_{k+1}=i_{0,2},\quad \ldots,\quad i_{k+n_{0}-1}=i_{0,n_{0}},
         \end{align*}
         \item for each $i$, every row in $A_{i}$ is non-zero, and
         \item \textcolor{red}{for each $P\in\mathcal{A}$, the sequence $(\phi_{i_1}\circ\phi_{i_2}\circ\cdots\circ\phi_{i_n}(\supp P))_n$ has
         the van Hove property.}
         \end{enumerate}
         Then, for any finite non-empty patch\/ $\mathcal{P}$, there
         is\/ $c_{\mathcal{P}}\in\R$ such that
         \begin{align*}
           \lim_{R\rightarrow\infty}\frac{1}{\mu_{\mathrm{L}}(B_R)}
           \card\bigl\{t\in B_R\mid \mathcal{P}+t\subset\mathcal{S}\bigr\}=c_{\mathcal{P}}
         \end{align*}
         converges uniformly for $\mathcal{S}\in \{\mathcal{T}+t\mid t\in\R^d\}$.
\end{thm}
\begin{proof}[Sketch of proof]
      This is the ``geometric'' version of the argument in \cite[Section 5.2]{Berthe-Delecroix}
      and the proof is similar.
\end{proof}

Note that the uniform convergence on the orbit $\{\mathcal{T}+t\mid t\in\R^d\}$ implies the uniform
convergence on the continuous hull, the closure of the orbit with respect to the
local matching topology.
Note also that for the single substitution case ($m_{a}=1$), the above conditions (1)-(3) for
the convergence of patch frequency are satisfied if the substitution matrix is primitive.

\subsection{Fourier transform, diffraction and the Lebesgue decomposition}
The diffraction measures associated with tilings are physically
important. They model the results of diffraction
experiments. Mathematically, the diffraction measure of a tiling is
the Fourier transform of the autocorrelation measure associated
with the tiling, described as follows.

In what follows, we have to deal with objects such as $\sum_{t\in
  D}c_{t}\delta_{t}$, where $D\subset\R^d$, $c_{t}\in\mathbb{C}$ and
$\delta_{t}$ is the Dirac (point) measure at $t$. 
We consider them as complex measures in the sense of \cite{Bourbaki_integration_chap1-4}
and call them Radon measures.
For a Radon measure $\mu$ on $\R^d$ and a function $\varphi\in L^1(\mu)$, we use a notation
\begin{align*}
      \langle\varphi,\mu\rangle=\int_{\mathbb{R}^d}\varphi d\mu.
\end{align*}

Let $C_{c}(\R^d)$ denote the vector space of all complex-valued,
continuous, compactly supported functions on $\R^d$. 
According to \cite{AdL-Fourier}, a Radon measure $\mu$ on $\R^d$ is said to
be \emph{Fourier transformable} if there is another Radon measure
\textcolor{red}{$\nu$} on $\R^d$ such that, for each $\varphi,\psi\in
C_{c}(\R^d)$, the inverse Fourier transform $\widecheck{\varphi*\psi}$
of the convolution of $\varphi$ and $\psi$ is in $L^{1}(\nu)$ and
\begin{align*}
\langle\varphi*\psi ,\mu\rangle=\left\langle\widecheck{\varphi*\psi} ,\nu\right\rangle
\end{align*}
If such a $\nu$ exists, it is unique, called the
\emph{Fourier transform} of $\mu$ and denoted by $\hat{\mu}$.  It is known that if $\mu$
is \emph{positive definite}, that is, if for each $\varphi\in
C_{c}(\R^{d}$) we have
\begin{align*}
     \langle \mu, \varphi*\tilde{\varphi}\rangle\geqq 0,
\end{align*}
then $\mu$ is Fourier transformable and the Fourier transform
$\hat{\mu}$ is positive \cite[Theorem
  4.11.5]{Baake-Grimm-vol2}.

Given a  Radon measure $\mu$, we define its \emph{diffraction measure}
as follows.
First, assume the following limit, the \emph{autocorrelation measure}, exists:
\begin{align}
       \mu\circledast\tilde{\mu}=
       \lim_{R\rightarrow\infty}\frac{\maprestriction{\mu}{B_{R}}*
                                  \maprestriction{\tilde{\mu}}{B_{R}}}
                                  {\mu_{\mathrm{L}}(B_{R})},\label{eq_def_autocorrelation}
\end{align}
where, for a Radon measure $\mu$  and a subset $S\subset\R$, the
restriction $\maprestriction{\mu}{S}$ is a Radon measure that sends
$\varphi\in C_{c}(\R)$ to $\int_{S}\varphi \, u\, d\mu$.  $\tilde{\mu}$
is defined via
$\langle\tilde{\mu},\varphi\rangle=\overline{\langle\mu,\widetilde{\varphi}\rangle}$,
where $\widetilde{\varphi}(t)=\overline{\varphi(-t)}$ for each $t\in\R^d$.
The limit \eqref{eq_def_autocorrelation} is nothing but a Radon measure
that sends $\varphi$ to
\begin{align*}
  \lim_{R\rightarrow\infty}\frac{1}{\mu_{\mathrm{L}}(B_{R})}\int_{B_{R}}\int_{B_{R}}\varphi(s-t)\,u(s)\,
  \overline{u(-t)}\,d\mu(s)\,d\mu(t),
\end{align*}
By construction, this limit
$\mu\circledast\tilde{\mu}$ is positive definite, and so its
Fourier transform exists and is positive.
We call this Fourier transform the \emph{diffraction measure} for $\mu$
\cite[Definition 9.2]{Baake-Grimm_vol1}.

In general, given a finite set $\mathcal{A}=\{T_1,T_{2},\ldots,
T_{n_{a}}\}$ of tiles and a tiling $\mathcal{T}$ in $\R^d$ whose tiles are
translates of elements of $\mathcal{A}$, we set \textcolor{red}{$D_{i}$} via
\begin{align*}
     D_i=\{t\in\mathbb{R}^d\mid T_i+t\in\mathcal{T}\}.
\end{align*}
 We then take complex numbers $w_{1},w_{2},\ldots, w_{n_{a}}$,
and consider a Radon measure
\begin{align*}
      \mu_{\mathcal{T}}= \sum_{i=1}^{n_{a}}w_{i}\sum_{t\in D_{i}}\delta_{t}.
\end{align*}
The diffraction measure for $\mu_{\mathcal{T}}$ is called the
\emph{diffraction measure for $\mathcal{T}$}.  It is easy to prove
that the autocorrelation measure is
\begin{align}
      \mu_{\mathcal{T}}\circledast\widetilde{\mu_{\mathcal{T}}}=
      \sum_{i,j=1}^{n_{a}}w_{i}\overline{w_{j}}\,
      \sum_{z\in\R}\freq_{\mathcal{T}}\{T_{j},T_{i}+z\}\,\delta_{z}.
      \label{autoco_measure_for_tiling}
\end{align}

The \emph{Lebesgue decomposition}\label{explanation_Lebesgue_dec} 
\cite[\S 5]{Bourbaki_integration_chap5}
of
a Radon measure is fundamental in the theory of diffraction. \textcolor{red}{In general,
a Radon measure $\mu$ on $\R^d$ is \emph{pure point} if its total variation $|\mu|$ is pure point, that is, a sum of Dirac measures.}
 \textcolor{red}{If $|\mu|(\{x\})=0$ for each $x\in\R^d$,} then $\mu$ is said to be \emph{continuous}. Any
Radon measure $\mu$ is uniquely decomposed into its pure point part
$\mu_{\text{pp}}$ and continuous part $\mu_{\text{c}}$.  The
continuous part is further decomposed into the singular continuous
component $\mu_{\text{sc}}$, which is mutually singular with the Lebesgue
measure on $\R^d$, and absolutely continuous component
$\mu_{\text{ac}}$, which is absolutely continuous with respect to the 
Lebesgue measure. Thus we have a decomposition
\begin{align}
      \mu=\mu_{\text{pp}}+\mu_{\text{sc}}+\mu_{\text{ac}}.\label{Lebesgue_dec_for_mu}
\end{align}
This decomposition of $\mu$ into its pure point, its continuous and
singular, and its continuous and absolutely continuous part is unique.

That
$\mu_\text{ac}$ is absolutely continuous means that there is a locally
integrable function $f\in L^{1}_{\text{loc}}(\R^d)$ such that
\begin{align*}
   \langle\mu_\text{ac},\varphi\rangle=\int\varphi \, f\, d\mu_{\mathrm{L}}
\end{align*}
holds for each $\varphi\in C_{c}(\R^d)$, where the right-hand side is
the integral with respect to the Lebesgue measure $\mu_{\mathrm{L}}$.  This $f$
is called the \emph{Radon--Nikodym derivative} of the Radon measure
$\mu$.

\section{Main results}
\label{section_main-result}
\subsection{A relation between the asymptotic behavior of Fourier matrices and absolutely continuous spectrum for S-adic tilings}
\label{subsection_sufficient-condition}
In this section, we prove a sufficient condition for the absence of the absolutely
continuous part of the diffraction measure for S-adic tilings. The
following setting is assumed for the whole section.

\begin{setting}\label{setting_relation_exponents_absconti-spec}
In this section, we take finite set\/ $\{\rho_{1},\rho_{2},\ldots, \rho_{m_{a}}\}$ of substitution rules \textcolor{red}{in $\R^d$} that share the
same arbitrary \textcolor{red}{(not necessarily $[0,1]^d$-supported)} alphabet\/ $\mathcal{A}=\{T_{1},T_{2},\ldots, T_{n_{a}}\}$.
\textcolor{red}{We assume the family $\{\rho_1,\rho_2,\ldots, \rho_{m_a}\}$ has FLC (page \pageref{def_FLC}).}
(Note that each substitution here is a ``geometric'' one and not a
``symbolic'' one.)  The existence of such common tiles and alphabet
 is the assumption
which we start with. Let\/ $\phi_{i}$ be the
expansion map for the substitution\/ $\rho_{i}$.  
(For different $i$ and $j$, the maps $\phi_i$ and $\phi_j$ may be different.)
The 
Fourier matrix for\/ $\rho_{i}$ is denoted by
 $B^{(i)}$, where\/
$B^{(i)}(t)=(B^{(i)}_{k,j}(t))_{k,j}$.

We consider a directive sequence\/ $(i_{j})^{}_{j=1,2,\cdots}$ \textcolor{red}{in $\{1,2,\ldots, m_a\}^{\mathbb{N}}$} and
let\/ $\mathcal{T}^{(1)}$ be an S-adic tiling that belongs to\/
$(i_{j})_{j}$. As we have seen on page
\pageref{explanation_de-substitution}, we have an increasing sequence\/
$n_{1}<n_{2}<\cdots$ of natural numbers and patches\/ $\mathcal{P}_{l}$
consisting of translates of alphabets such that
\begin{align*}
  \mathcal{T}^{(k)}&=\lim_{l\rightarrow\infty}
  \rho_{i_{k}}\circ\rho_{i_{k+1}}\circ\cdots\circ\rho_{i_{n_{l}}}(\mathcal{P}_{l})\\
     &=\lim_{l\rightarrow\infty}\rho_{i[k,n_{l})}(\mathcal{P}_{l})
\end{align*}
converges for each\/ $k=1,2,\cdots$.  Note that we do not assume
recognizability here, but we do assume that, for each\/
$\mathcal{T}^{(k)}$, the patch frequencies converge.
\end{setting}

Since each $\rho_{i}$, regarded as a map that sends a patch
$\mathcal{P}$ to another patch $\rho_{i}(\mathcal{P})$, is continuous
with respect to the local matching topology, we see that
\begin{align*}
      \rho_{i_{k}}(\mathcal{T}^{(k+1)})=\mathcal{T}^{(k)}
\end{align*}
for each $k=1,2,\ldots$.
This can be used to ``compare'' \textcolor{red}{the autocorrelation measure for $\mathcal{T}^{(k)}$ and one for
$\mathcal{T}^{(k+1)}$.}

The fundamental idea to study the diffraction spectrum is to use
renormalization equations
 \cite{BFGR, Baake-Gahler, Baake-Grimm_renormalisation, 
 Baake-Grimm-Manibo, Manibo_binary, Manibo_thesis}.
 \textcolor{red}{The above ``de-substitution'' or inverse-limit structure gives us a renormalization scheme,
 which in turn gives us a sufficient condition for zero absolutely continuous spectrum in terms of
 an asymptotic of norms of Fourier matrices (Theorem \ref{thm_liminf_implies_zero_acpart}).
  Such an asymptotic behavior can be checked by ergodic theory as in
 Section \ref{subsection_absence_acpart}.
 The special case of Theorem \ref{thm_liminf_implies_zero_acpart} for the substitution case (the case where $m_a=1$)
 was proved by Ma\~nibo \cite{Manibo_binary, Manibo_thesis}.
 Below we adapt Ma\~nibo's idea to the general S-adic case.}

\textcolor{red}{
The goal of this section is to prove the following theorem.}
\begin{thm}\label{thm_liminf_implies_zero_acpart}
        \textcolor{red}{If there is\/ $\e>0$ such that
        \begin{align}
               \liminf_{k\rightarrow\infty}\Bigl(&\frac{1}{2k}\log\det\phi_{i_{1}}\det\phi_{i_{2}}\cdots\det\phi_{i_{k}}\nonumber\\
               &-\frac{1}{k}\log\bigl\|B^{(i_{1})}(t)B^{(i_{2})}(\phi^*_{i_{1}}(t))\cdots B^{(i_{k})}(\phi^*_{i_{k-1}}\circ\phi^*_{i_{k-2}}\circ\cdots\circ\phi^*_{i_{1}}(t))\bigr\|\Bigr)>\e
               \label{eq_liminf}
        \end{align}
        for Lebesgue-a.e.\/ $t\in\R$, where ${}^*$ denotes the adjoint, then the diffraction
        spectrum of\/ $\mathcal{T}^{(1)}$ has zero absolutely
        continuous part.}
\end{thm}

\begin{rem}
      There is a similar result \textcolor{red}{for the one-dimensional case} for the corresponding dynamical
      spectrum in a recent paper by Bufetov and Solomyak \cite[Corollary
        4.5]{Bufetov-Solomyak_singular}.  
        They proved a sufficient condition for
        the absence of absolutely continuous dynamical spectrum for
        suspension flows for S-adic sequences.
        This covers some cases which Theorem \ref{thm_liminf_implies_zero_acpart}  does not
        cover, but Theorem \ref{thm_liminf_implies_zero_acpart}  deals with some cases
        which Bufetov and Solomyak did not.
        The sufficient condition in  \cite{Bufetov-Solomyak_singular} is similar to Theorem 
        \ref{thm_liminf_implies_zero_acpart},
        but they replace $\frac{1}{2k}\log\lambda_{i_{1}}\lambda_{i_{2}}\cdots\lambda_{i_{k}}$
        with $\frac{1}{2k}\log\|A_{i_{1}}A_{i_{2}}\cdots A_{i_{k}}\|$ 
        ($A_{i}$ is the substitution matrix for $\rho_{i}$)
        and assume that the limit of these
        as $k\rightarrow\infty$ is convergent. Moreover, they assume the recognizability for the
        directive sequence $i_{1},i_{2},\ldots$.
        Therefore, Theorem
      \ref{thm_liminf_implies_zero_acpart} covers some cases that
      Bufetov and Solomyak did not cover, since \textcolor{red}{
      the dimension $d$ of the tiling is arbitrary and} the directive sequence is arbitrary in 
      this theorem.
      On the other hand, Bufetov and Solomyak deal with arbitrary suspension flows for
      S-adic sequences, whereas in this paper, \textcolor{red}{for the one-dimensional cases,} we only deal with tile lengths that come from
      \textcolor{red}{Perron--Frobenius} eigenvector.
\end{rem}

\textcolor{red}{
For the rest of this section we will prove Theorem \ref{thm_liminf_implies_zero_acpart}. 
The readers may skip the proof for the first reading and
move to an application of this theorem in Section \ref{subsection_absence_acpart}.}

\begin{defi}
    Let\/ $D^{(k)}_{i}$ be defined via
    \begin{align*}
           D^{(k)}_i=\{t\in\R^d\mid T_i+t\in\mathcal{T}^{(k)}\}.
    \end{align*}
 The density of
each $D^{(k)}_{i}$ is defined via
\begin{align*}
    \dens D^{(k)}_{i}=\lim_{R\rightarrow\infty}\frac{1}{\mu_{\mathrm{L}}(B_R)}\card D^{(k)}_{i}\cap B_R,
\end{align*}
where the limit is convergent since the patch frequencies converge.

For each\/ $k=1,2,\cdots$, each $i,j=1,2,\ldots, n_{a}$ and each\/
$z\in\R^d$, we set
\begin{align*}
    \nu^{(k)}_{i,j}(z)&=\freq_{\mathcal{T}^{(k)}} \{T_{i}, T_{j}+z\}\\
                    &=\lim_{R\rightarrow\infty}\frac{1}{\mu_{\mathrm{L}}(B_R)}
                      \card\bigl\{t\in B_R\mid T_{i}+t, T_{j}+t+z\in\mathcal{T}^{(k)}\bigr\}.
\end{align*}
\end{defi}

\begin{lem}\label{lem_for_renormalization}
       For each\/ $k$, we have the following equation:
       \begin{align}
              \frac{1}{\det\phi_{i_{k}}}\sum_{m,n=1}^{n_{a}}
              \sum_{x\in T^{(i_{k})}_{i,m}}\sum_{y\in T^{(i_{k})}_{j,n}}
              \nu_{m,n}^{(k+1)}\bigl(\phi_{i_k}^{-1}(z+x-y)\bigr)=\nu^{(k)}_{i,j}(z).
       \end{align}
\end{lem}
\begin{proof}
    The proof is essentially the same as the substitution case (that
    is, the case where $m_{a}=1$) in \cite{Manibo_thesis}, but since
    the proof in \cite{Manibo_thesis} appears to use recognizability,
    we here give an outline of the proof without using
    recognizability.
    
    Take $i,j,k$ and $z\in D^{(k)}_{j}-D^{(k)}_{i}$. For each $t\in
    D^{(k)}_{i}$ such that $t+z\in D^{(k)}_{j}$, since
    $T_{i}+t\in\mathcal{T}^{(k)}$ and $T_{j}+t+z\in\mathcal{T}^{(k)}$,
    there are $m,n\in\{1,2,\ldots, n_{a}\}$, $s_{m}\in D^{(k+1)}_{m}$
    and $s_{n}\in D^{(k+1)}_{n}$ such that
    $T_{i}+t\in\rho_{i_{k}}(T_{m}+s_{m})$ and
    $T_{j}+t+z\in\rho_{i_{k}}(T_{n}+s_{n})$.  (These $m,n,s_{m},s_{n}$
    are unique.)  By computation, it follows that
    \begin{align*}
          s_{n}-s_{m}=\phi_{i_k}^{-1}(z+x-y)
    \end{align*}
    for some $x\in T^{(i_{k})}_{i,m}$ and $y\in T^{(i_{k})}_{j,n}$, and so we obtain a map that sends
    \begin{align}
      t\in\bigl\{t\in\R^d\mid T_{i}+t, T_{j}+t+z\in\mathcal{T}^{(k)}\bigr\}
      \label{eq1_renormalization}
    \end{align}
    to
    \begin{align}
           s_{m}\in\!\bigsqcup_{m,n\leqq n_{a}}\bigsqcup_{x\in T^{(i_{k})}_{i,m}}
           \bigsqcup_{y\in T^{(i_{k})}_{j,m}}\bigl\{s\in\R^d\mid T_{m}+s\in\mathcal{T}^{(k+1)},
           T_{n}+s+\phi_{i_k}^{-1}(z+x-y)\in\mathcal{T}^{(k+1)}\bigr\},
           \label{eq2_renormalization}
    \end{align}
    where $\bigsqcup$ means taking the union while regarding the sets as disjoint.
    We can show that this map is a bijection.
    
   We then see that
   \begin{align*}
         & \sum_{m,n\leqq n_{a}}\sum_{x\in T^{(i_{k})}_{i,m}}
     \sum_{y\in T^{(i_{k})}_{j,m}}\card\bigl\{s\in \phi_{i_k}^{-1}(B_{R-C'})\mid 
                      T_{m}+s\in\mathcal{T}^{(k+1)},
                      T_{n}+s+\phi_{i_k}^{-1}(x-y+z)\in\mathcal{T}^{(k+1)}\bigr\}\\
                    &  \leqq \card\bigl\{t\in B_{R}\mid T_{i}+t, T_{j}+t+z\in\mathcal{T}^{(k)}\bigr\}\\
                     & \leqq \sum_{m,n\leqq n_{a}}\sum_{x\in T^{(i_{k})}_{i,m}}
                      \sum_{y\in T^{(i_{k})}_{j,m}}\card\bigl\{s\in \phi_{i_k}^{-1}(B_R)+B_C\mid 
                      T_{m}+s\in\mathcal{T}^{(k+1)},
                      T_{n}+s+\phi_{i_k}^{-1}(z+x-y)\in\mathcal{T}^{(k+1)}\bigr\},
   \end{align*}
   where  $C'$ is the maximal norm of vectors in all digits for $\rho_1,\rho_2,\ldots, \rho_{m_a}$ and
   $C$ is the maximal diameter for the elements of the alphabet.
   By dividing these by $\mu_{\mathrm{L}}(B_R)$ and
   taking the limit as $R\rightarrow\infty$, we obtain the desired
   equation.
\end{proof}

\begin{defi}
      For each $k=1,2,\cdots$, $i,j=1,2,\ldots, n_{a}$, we set
      \begin{align*}
           \Upsilon^{(k)}_{i,j}=\sum_{z\in D^{(k)}_{j}-D^{(k)}_{i}}\nu_{i,j}^{(k)}(z)\,\delta_{z} .
      \end{align*}
      \textcolor{red}{By FLC}, for each compact set $K$
     the set $K\cap(D^{(k)}_{j}-D^{(k)}_{i})$ is a finite set, and so the infinite sum
     $\sum_{z\in D^{(k)}_{j}-D^{(k)}_{i}}$ (with respect to the vague topology)
      is well-defined.
\end{defi}

Note that, by equation \eqref{autoco_measure_for_tiling}, it suffices
to investigate these $\Upsilon^{(k)}_{i,j}$ in order to understand
the diffraction measure for $\mathcal{T}^{(k)}$.
The Fourier transforms of the $\Upsilon^{(k)}_{i,j}$ will generate the
diffraction measures for $\mathcal{T}^{(k)}$. In order to investigate
the nature of these diffraction measures, it is useful to use the
relation between $\Upsilon^{(k)}_{i,j}$ ($i,j=1,2,\ldots, n_{a}$) and
$\Upsilon^{(k+1)}_{m,n}$ ($m,n=1,2,\ldots, n_{a}$) stated in
Proposition \ref{prop_renormalization} below.  In order to give a
statement, we first introduce two symbols.

\begin{defi}
      If a Radon measure $\mu$ on $\R^d$ and a homeomorphism $g\colon\R^d\rightarrow\R^d$ are given, 
      define another Radon measure $g.\mu$ via
      \begin{align*}
            \langle g.\mu,\varphi\rangle=\langle\mu, \varphi\circ g\rangle
      \end{align*}
      for each $\varphi\in C_{c}(\R^d)$.
      
      Let $\mu$ be a Radon measure on $\R^d$ and let $\xi$ be a continuous function on $\R^d$.
      The new Radon measure $\xi\mu$ is defined via
      \begin{align*}
            \langle \xi\mu,\varphi\rangle=\langle\mu, \xi\varphi\rangle
      \end{align*}
      for each $\varphi\in C_{c}(\R^d)$, where $\xi\varphi$ is the pointwise multiplication.
\end{defi}

In what follows, we use \emph{convolutions} of Radon measures
\cite[Definition 4.9.18]{Baake-Grimm-vol2}.

\begin{prop}[Renormalization equation]\label{prop_renormalization}
      For each\/ $k,i,j$,
      we have
      \begin{align}
            \Upsilon^{(k)}_{i,j}=
            \frac{1}{\det\phi_{i_k}}\sum_{m,n=1}^{n_{a}}
            \sum_{x\in T^{(i_{k})}_{i,m}}\sum_{y\in T^{(i_{k})}_{j,n}}
            \delta_{y-x}*\bigl(\phi_{i_{k}}.\Upsilon^{(k+1)}_{m,n}\bigr).
      \end{align}
\end{prop}
\begin{proof}
    This can be proved using Lemma \ref{lem_for_renormalization} by a direct computation.
\end{proof}

This renormalization equation \textcolor{red}{gives} rise to an equation between Fourier
transforms, \textcolor{red}{as in Proposition \ref{prop_renormalization_in_dual}.}
In what follows, for each $t\in\R^d$, the symbol $\exp_{t}$ denotes the exponential function defined via
$\exp_{t}(s)=e^{2\pi \im \langle s,t\rangle}$ for each $s\in\R^d$, where
$\langle\cdot,\cdot\rangle$ is the standard inner product.

\begin{prop}\label{prop_renormalization_in_dual}
     Each\/ $\Upsilon^{(k)}_{i,j}$ is Fourier transformable and
          we have
     \begin{align}
           \widehat{\Upsilon^{(k)}_{i,j}}
           =\frac{1}{(\det\phi_{i_k})^2}
             \sum_{m,n=1}^{n_{a}}\sum_{x\in T^{(i_{k})}_{i,n}}\sum_{y\in T^{(i_{k})}_{j,n}}
             \exp_{x-y}(\phi^*_{i_{k}})^{-1}.\widehat{\Upsilon^{(k+1)}_{m,n}}.
             \label{eq_renormalization_fourier-dual}
     \end{align}
              where\/ ${}^{*}$ denotes the adjoint map.
\end{prop}
\begin{proof}
       To prove that each $\Upsilon^{(k)}_{i,j}$ is Fourier transformable, fix $k$ and
       set $\omega_{i}=\sum_{x\in D^{(k)}_{i}}\delta_{x}$ for each
       $i=1,2,\ldots, n_{a}$.
       By computation we have
       \begin{align*}
              \Upsilon^{(k)}_{i,j}&=\omega_{i}\circledast\widetilde{\omega_{j}},
       \end{align*}
       which is the sum of four positive definite Radon measures by the
       polarization identity.  Since each positive definite Radon
       measure is Fourier transformable, the sum $\Upsilon^{(k)}_{i,j}$
       is also Fourier transformable.  The formula
       \eqref{eq_renormalization_fourier-dual} is a direct consequence
       of Proposition \ref{prop_renormalization}.
\end{proof}

We can decompose each $\Upsilon^{(k)}_{i,j}$ into its pure point,
absolutely continuous and singular continuous component (with respect to the
Lebesgue measure, see page \pageref{explanation_Lebesgue_dec}):
\begin{align*}
      \Upsilon_{i,j}^{(k)}=(\Upsilon^{(k)}_{i,j})_\text{pp}+(\Upsilon^{(k)}_{i,j})_\text{ac}+(\Upsilon^{(k)}_{i,j})_\text{sc}.
\end{align*}
The Radon--Nikodym derivative of the absolutely continuous part
$(\Upsilon^{(k)}_{i,j})_\text{ac}$ is denoted by $h^{(k)}_{i,j}\in
L^{1}_{\text{loc}}(\R^d)$.

By the uniqueness of the Lebesgue decomposition and by Proposition
\ref{prop_renormalization_in_dual}, we have the following result.

\begin{prop}\label{prop_renormalization_RNderivative}
  The Radon--Nikodym derivative\/ $h^{(k)}_{i,j}$ satisfies
  \begin{align}
      h^{(k)}_{i,j}(t)=\frac{1}{\det\phi_{i_k}}
                     \sum_{m,n=1}^{n_{a}}\sum_{x\in T^{(i_{k})}_{i,m}}\sum_{y\in T^{(i_{k})}_{j,n}}
                           \exp_{x-y}(t)\,h^{(k+1)}_{m,n}(\phi^*_{i_{k}}(t))
     \end{align}
     for Lebesgue-a.e.\/ $t\in\R^d$.
     \end{prop}

\textcolor{red}{We will investigate when $h^{(1)}_{i,j}$ are zero and so the absolutely continuous spectrum
for $\mathcal{T}^{(1)}$ is zero.}
It is convenient to define the following Radon--Nikodym matrix.

\begin{defi}\label{def_R-Nmatrix}
      For each $k$, define the \emph{Radon-Nikodym matrix} $H^{(k)}$
      for the tiling $\mathcal{T}^{(k)}$, which is a matrix-valued
      function on $\R^d$, via
      \begin{align*}
              H^{(k)}(t)=(h^{(k)}_{i,j}(t))_{i,j}.
      \end{align*}
\end{defi}

Then, Proposition \ref{prop_renormalization_RNderivative} now translates
into the following result.
\begin{prop}
          For each\/ $k$, we have
          \begin{align}
                 H^{(k)}(t)=\frac{1}{\det\phi_{i_{k}}}
                                      B^{(i_{k})}(t)\,H^{(k+1)}(\phi^*_{i_{k}}(t))\,B^{(i_{k})}(t)^{*}.
                                      \label{renormaliztion_for_RNmatrix}
          \end{align}
\end{prop}

In Section \ref{subsection_absence_acpart}, we will prove that these
$H^{(k)}(t)$ are in fact zero for certain choices of
$\{\rho_{1},\rho_{2},\ldots, \rho_{m_{a}}\}$. For this purpose, the
following lemma is useful, because a positive definite matrix is zero
if its trace is zero.

\begin{lem}\label{lem_positive_definite}
        For each\/ $k$, the matrix\/ $H^{(k)}(t)$ is a positive definite
        matrix for Lebesgue-a.e.\/ $t\in\R^d$.
\end{lem}
\begin{proof}
        By the definition of a positive definite matrix, we have to
        prove that, whenever we take $w_{1},w_{2},\ldots
        ,w_{n_{a}}\in\C$, we have
        \begin{align}
               \sum_{i,j}w_{i}\overline{w_{j}}h^{(k)}_{i,j}(t)\geqq 0.
        \end{align}
        For $w_{j}$ which are arbitrarily taken from $\C$, set
        \begin{align*}
             \gamma=\sum w_{j}\sum_{t\in  D^{(k)}_{j}}\delta_t.
        \end{align*}
        By \eqref{autoco_measure_for_tiling}, the autocorrelation
        measure for $\gamma$ is positive definite and coincides with
        \begin{align*}
                 \sum_{i,j}w_{i}\,\overline{w_{j}}\,\Upsilon^{(k)}_{j,i}
        \end{align*}
        Since the Fourier transform of this Radon charge is positive,
        so is its Radon--Nikodym derivative, which is
        \begin{align*}
                \sum_{i,j}w_{i}\,\overline{w_{j}}\, h^{(k)}_{j,i},
        \end{align*}
        by the uniqueness of the Lebesgue decomposition.
\end{proof}


\begin{defi}\label{def_lambda(l)_overlineB(l)}
        For each $k\geqq 1$, set
        \begin{align*}
             \lambda^{(k)}=\det\phi_{i_{1}}\det\phi_{i_{2}}\cdots\det\phi_{i_{k}},
        \end{align*}
        and
        \begin{align*}
                      \phi_{(k)}=\phi_{i_1}\circ\phi_{i_{2}}\circ\cdots\phi_{i_k}.
        \end{align*}
        and set $\lambda^{(0)}=1$.
        For each $k\geqq 1$ and $t\in\R^d$, define
        \begin{align*}
          \overline{B}^{(k)}(t)=
          B^{(i_{1})}(t)\, B^{(i_{2})}(\phi_{(1)}^*(t))\, \cdots\, B^{(i_{k})}(\phi_{(k-1)}^*(t)).
        \end{align*}
\end{defi}

By iterating the equation \eqref{renormaliztion_for_RNmatrix}, we have:
\begin{prop}\label{prop_estimate_trace}
       For each\/ $k\geqq 1$, we have
       \begin{align}
         H^{(1)}(t)=\frac{1}{\lambda^{(k)}}\,\overline{B}^{(k)}(t)\,
         H^{(k+1)}(\phi_{(k)}^*(t))\, \overline{B}^{(k)}(t)^{*}
       \end{align}
       for Lebesgue-a.e.\/ $t\in\R$. Hence we have the following inequality for the traces:
       \begin{align}
            \Tr(H^{(1)}(t))
            \leqq\frac{\|\overline{B}^{(k)}(t)\|^{2}}{\lambda^{(k)}}\Tr(H^{(k+1)}(\phi_{(k)}^*(t)).\label{ineq_for_absence_acpart}
       \end{align}
\end{prop}

By Proposition \ref{prop_estimate_trace}, we have the following
strategy to prove that $\Tr(H^{(1)}(t))=0$ and hence $H^{(1)}(t)=0$ 
($H^{(1)}(t)$ is a positive definite matrix) and the absence of an
absolutely continuous component of the diffraction measure for
$\mathcal{T}^{(1)}$.

First, we prove that (i) the traces $\Tr(H^{(k+1)}(\phi^*_{(k)}t))$
are, in a sense, ``bounded from above''.  Next, we prove that (ii) the
convergence
\begin{align}
  \frac{\|\overline{B}^{(k)}(t)\|^{2}}{\lambda^{(k)}}\rightarrow 0\quad
  \text{as $k\rightarrow\infty$}
   \label{convergence_Bk_over_lambda}
\end{align}
implies that $\Tr(H^{(1)}(t))=0$, using (i) and the inequality \eqref{ineq_for_absence_acpart}.

We now proceed to prove (i) (Lemma \ref{lem_existence_A}) and (ii) (Theorem
\ref{thm_liminf_implies_zero_acpart}). The remaining part, condition
\eqref{convergence_Bk_over_lambda}, is proved for a special class of
S-adic tilings in Section \ref{subsection_absence_acpart}.

\begin{lem}\label{lem_existence_A}
    There are $A>0, k_0>0, R_0>0$ such that, for
    any natural number $k\geqq k_0$ and any real number $R>R_0$, we have
    \begin{align}
            \int_{[-R,R]^d}\Tr(H^{(k)}(\phi_{(k-1)}^*(t)))\, dt\leqq AR^d.
    \end{align}
\end{lem}

\begin{proof}
%
%
                  By \cite[Corollary 4.9.12]{Baake-Grimm-vol2}, there is an $f\in
         C_{c}(\R^d)$ such that $\widecheck{f*\widetilde{f}}\geqq
         1_{[0,1]^d}$, where $1_{S}$ for $S\subset\R^d$ is the characteristic function
         for $S$.  By computation, we have
         \begin{align*}
                1_{\phi_{(k-1)}^*([0,1]+t)}\leqq (T_t(\widecheck{f*\widetilde{f}}))\circ(\phi_{(k-1)}^*)^{-1},
         \end{align*}
         where $T_t$ denotes the translation by $t$.
         

                  If $k$ is large enough, $z=0$ is the only element $z$ in $D^{(k)}_i-D^{(k)}_i$ such that
         $\phi_{(k-1)}(z)\in\supp f*\widetilde{f}$.    
         For any $k>0$ which is large enough in this sense, $i=1,2,\ldots, n_a$ and $t\in\R^d$, we have
         \begin{align*}
                \langle 1_{\phi_{(k-1)}^*([0,1]+t)},\widehat{\Upsilon_{i,i}^{(k)}}\rangle
                &\leqq\det\phi_{(k-1)}
                \langle\exp_{-t}(f*\widetilde{f})\circ\phi_{(k-1)} ,\Upsilon_{i,i}^{(k)}\rangle\\
                &\leqq \det\phi_{(k-1)}f*\widetilde{f}(0)\nu^{(k)}_{i,i}(0).
         \end{align*}
         
         Note that $\nu^{(k)}_{i,i}(0)=\dens D^{(k)}_i$. Set $B=\sup_{k,i}\dens D^{(k)}_i$.
         This is finite since $\mathcal{A}$ is fixed.
         For an arbitrary choice of $R>0$, if $n$ is a natural number with $n\leqq R<n+1$, we have:
         \begin{align*}
                \int_{[-R,R]^d}\Tr(H^{(k)}(\phi^*_{(k-1)}(t))dt
                &\leqq\sum_{i=1}^{n_a}\sum_{v\in\mathbb{Z}^d\cap[-n-1,n]^d}\frac{1}{\det\phi_{(k-1)}}\left\langle\widehat{\Upsilon^{(k)}_{i,i}}, 1_{\phi^*_{(k-1)}([0,1]^d+v)}\right\rangle\\
                &\leqq n_a(2R+2)^df*\widetilde{f}(0)B.
         \end{align*}
         
         By taking $A>2^dn_a f*\widetilde{f}(0)B$, we have the conclusion.  
%
\end{proof}

We now prove a sufficient condition for $H^{(1)}$ to be zero.  As we
stated before, the condition is the exponential convergence of
$\frac{\|\overline{B}^{(k)}(t)\|^{2}}{\lambda^{(k)}}\rightarrow 0$ as
$k\rightarrow\infty$.  We will prove that this sufficient condition is
actually satisfied for a class of S-adic tilings, in Section
\ref{subsection_absence_acpart}, and thus prove the absence of
absolutely continuous part of the diffraction spectrum for this class.

\begin{thm}\label{thm_liminf_implies_zero_H}
        If there is\/ $\e>0$ such that
        \begin{align}
               \liminf_{k\rightarrow\infty}\left(\frac{1}{2k}\log\lambda^{(k)}
               -\frac{1}{k}\log\bigl\|\overline{B}^{(k)}(t)\bigr\|\right)>\e
               \label{eq2_liminf}
        \end{align}
        for Lebesgue-a.e.\/ $t\in\R$, then we have\/ $H^{(1)}=0$ for
        Lebesgue-a.e.\/ $t\in\R$, where\/ $H^{(1)}$ is the
        Radon--Nikodym matrix for\/ $\mathcal{T}^{(1)}$ as defined in
        Definition \ref{def_R-Nmatrix}.  
 \end{thm}

\begin{proof}
         Take an arbitrary
          positive real number $R$. For each $k\geqq 1$, set
          \begin{align*}
                 E_{k}=\bigl\{t\in [-R,R]^d\mid\inf_{l\geqq k}\bigl(\tfrac{1}{2l}\log\lambda^{(l)}-\tfrac{1}{l}\log\|\overline{B}^{(l)}(t)\|\bigr)\geqq\tfrac{\e}{2}\bigr\}.
          \end{align*}
          The sets $E_{1},E_{2},\ldots$ are increasing and $\bigcup_{k\geqq 1}E_{k}$ has
          full measure $(2R)^d$.

          For each $k\geqq 1$ and $t\in E_{k}$, we have
          \begin{align*}
                  \frac{\|\overline{B}^{(k)}(t)\|^{2}}{\lambda^{(k)}}\leqq e^{-k\e},
          \end{align*}
          and
          \begin{align*}
            \int_{[-R,R]^d}\Tr(H^{(1)}(t))\, dt
            &=\int_{E_{k}}\Tr(H^{(1)}(t))\, dt+\int_{[-R,R]^d\setminus E_{k}}\Tr(H^{(1)}(t))\, dt\\
                 &\leqq\int_{E_{k}}\frac{\|\overline{B}^{(k)}(t)\|^{2}}{\lambda^{(k)}}\,\Tr(H^{(k+1)}(\phi_{(k)}^*(t)))\, dt
                         +\int_{[-R,R]^d\setminus E_{k}}\Tr(H^{(1)}(t))\, dt.
          \end{align*}
          Here, if $R$ is large enough, the first term tends to zero as $k\rightarrow\infty$, since there exists $A>0$ as in
          Lemma \ref{lem_existence_A} and since
          \begin{align*}
                  \int_{E_{k}}\frac{\|\overline{B}^{(k)}(t)\|^{2}}{\lambda^{(k)}}\, \Tr(H^{(k+1)}(\phi_{(k)}^*(t)))\, dt
                  &\leqq e^{-k\e}\int_{[-R,R]^d}\Tr(H^{(k+1)}(\phi_{(k)}^*t))\, dt\\
                  &\leqq e^{-k\e}AR^d.
          \end{align*}
          The second term also tends to zero because  $\bigcup_{k\geqq 1}E_{k}$ has
          full measure $(2R)^d$ and 
          \begin{align*}
                \int_{[-R,R]^d}\Tr(H^{(1)}(t))\, dt<\infty.
          \end{align*}
          We see that
          \begin{align*}
                 \int_{[-R,R]^d}\Tr(H^{(1)}(t))\,dt=0,
          \end{align*}
          which, by  Lemma \ref{lem_positive_definite},
          implies that $\Tr(H^{(1)}(t))=0$ a.e., which in turn implies that $H^{(1)}(t)=0$ a.e.,
          since all the eigenvalues are zero and $H^{(1)}(t)$ is diagonalizable.
\end{proof}

\begin{proof}[The proof for Theorem \ref{thm_liminf_implies_zero_acpart}]
\textcolor{red}{
          The absence of the absolutely continuous part of the diffraction spectrum for $\mathcal{T}^{(1)}$
          follows from equation \eqref{autoco_measure_for_tiling}.}
\end{proof}

\begin{rem}
       \textcolor{red}{If} the patch frequency converges uniformly, the corresponding dynamical system for
       $\calT^{(1)}$ is uniquely ergodic, and so by \cite[Theorem 5]{Baake-Lenz}, the absence
       of absolutely continuous diffraction spectrum holds for any tiling in the continuous hull
       of $\mathcal{T}^{(1)}$.
\end{rem}

\subsection{The absence of absolutely continuous spectrum for binary block-substitution case}
\label{subsection_absence_acpart}
In this section, we deal with S-adic tilings with binary \textcolor{red}{block}
substitution rules.  The following setting is assumed for the whole
section.

\begin{setting}\label{setting_binary_const_length}
        In this section,  let\/ $\mathcal{A}=\{T_{1},T_{2}\}$ consist of two tiles \textcolor{red}{in $\R^d$ with\/
         $\supp T_{j}=[0,1]^d$ for each\/ $j$.}
        
        We take a finite family of substitutions\/
        $\rho_{1},\rho_{2},\ldots, \rho_{m_{a}}$ with alphabet\/
        $\mathcal{A}$.  \textcolor{red}{We assume these are block substitutions.
        In other words, the expansion map $\phi_i$ for $\rho_{i}$
           is defined by a diagonal matrix with integer diagonal entries.
         $B^{(i)}(t)$ is
        the Fourier matrix for\/ $\rho_{i}$.}
       
       We assume that, for each\/ $i$, there is a $t\in\R^d$ such that
       the Fourier matrix\/ $B^{(i)}(t)$ is not singular.  This means
       that\/ $B^{(i)}(t)$ is not singular for almost every\/
       $t\in\R^d$, since the determinant is a sum of exponential
       functions.
       
       Next, we take a measure-preserving system\/
       $(X,\mathcal{B},\mu,S_{0})$, where\/ $(X,\mathcal{B},\mu)$ is a
       standard probability space and\/ $S_{0}\colon X\rightarrow X$ is a
       measure-preserving transformation.  We assume that the map $S_{0}$
       is ergodic and surjective.
       
       We choose a decomposition
       \begin{align*}
               X=\bigcup_{j=1}^{m_{a}}E_{j}
       \end{align*}
       of\/ $X$ into pairwise disjoint subsets\/ $E_{j}\in\mathcal{B}$,
       $j=1,2,\ldots, m_{a}$.
\end{setting}


\textcolor{red}{The goal of this section is to prove Theorem \ref{thm_absence_for_binary}.}
In Theorem \ref{thm_absence_for_binary}, we will show the
absence of the absolutely continuous component in the diffraction
spectrum for S-adic tilings constructed from $\rho_1,\rho_2,\ldots,\rho_{m_a}$ obtained from ``almost all'' directive
sequences given by the coding of the system $(X,S_{0})$, by using Theorem
\ref{thm_liminf_implies_zero_acpart}.  
Specifically, we prove that there is \textcolor{red}{some} $E\subset X$ with measure $1$ such that, for each $x\in E$, the directive sequence $i_{1},i_{2},\ldots$ defined by
\begin{align*}
        S_{0}^{n-1}(x)\in E_{i_{n}},
\end{align*}
for each $n$, gives rise to an S-adic tiling with zero absolutely continuous diffraction spectrum.
(We prove that any S-adic tiling belonging to such a directive sequence has this property.)
\textcolor{red}{For example, if $X\subset\{1,2,\ldots,m_a\}^{\mathbb{N}}$ is a subshift with surjective shift-map and 
$E_j=\{x\in X\mid x_1=j\}$, then the directive sequence defined by $x\in X$ is $x$ itself.
The theorem says that for almost all $x\in X$, the S-adic tilings belonging to $x$ have
zero absolutely continuous spectrum.
Special cases for the theorem are stated in Introduction.
The readers can skip the detail of this section
 and jump to Theorem \ref{thm_absence_for_binary} and Example \ref{example_absense_acpart_binary}
for the first reading.}

\textcolor{red}{
Note that  since each $\rho_i$ is a block substitution, 
 the digit $T^{(i)}_{k,l}$ for $\rho_i$ is inside $\mathbb{Z}^d$, and moreover for each $k$ and $i$,}
        \begin{align*}
               T^{(i)}_{1,k}\cup T^{(i)}_{2,k}=\phi_i([0,1)^d)\cap\mathbb{Z}^d.
        \end{align*}
        \textcolor{red}{We denote this set by $F_i$.}

Theorem
\ref{thm_liminf_implies_zero_acpart} applies to the Fourier matrices
$B^{(i)}(t)$, but we will replace $B^{(i)}(t)$ with the following
matrices $C^{(i)}(z)$, where $z\in\T^d$: the latter has the
advantage of a compact domain.

\begin{defi}
        For each $i=1,2,\ldots, m_{a}$ and $t\in\R^d$, set
        \begin{align*}
                C^{(i)}(\pi(t))=B^{(i)}(t).
        \end{align*}
        ($\pi\colon\R^d\rightarrow\T^d$ is the quotient map.) \textcolor{red}{
        Since  the elements in the digits are all in $\mathbb{Z}^d$,}
         this is well-defined.
        
        \textcolor{red}{For a technical reason, we replace the system $(X,\mathcal{B},\mu,S_0)$ with its natural extension.}
        Let $(Y,\mathcal{C},\nu,S_{1})$ be a natural extension of $(X,\mathcal{B},\mu,S_{0})$.
        In other words, $(Y,\mathcal{C},\nu)$ is a separable and complete probability space
        and $S_{1}\colon Y\rightarrow Y$ is an invertible measure-preserving transformation, and
        there exists
        a factor map $f\colon Y\rightarrow X$. Such a system and a factor map exist since
        $S_{0}$ is surjective, and $S_{1}$ is ergodic \cite[Section 1.6.3]{Sarig}.        
        
        \textcolor{red}{The main tool to prove Theorem \ref{thm_absence_for_binary} is Furstenberg--Kesten theorem
        and Oseledets ergodic theorem. For these theorems, see for example \cite{Viana}.
        We will use these theorems to the following cocycle and obtain an estimate \eqref{eq_liminf}.}
        We define a map $C\colon\mathbb{T}^d\times Y\rightarrow GL_{2}(\mathbb{C})$ via
        \begin{align*}
               C(z,x)=
               \begin{cases}
                      C^{(i)}(z)^{-1}&\text{ if $x\in f^{-1}(E_{i})$ and $C^{(i)}(z)$ is invertible,}\\
                      I      &\text{ otherwise}.
               \end{cases}
        \end{align*}
        ($I$ is the identity matrix, but this does not matter since $C^{(i)}(z)$ is almost surely
        invertible.)
        
        We set $R\colon\T^d\times Y\rightarrow\T^d\times Y$ via
        \begin{align*}
               R(\pi(t),x)=\bigl(\pi(\phi_{i}(t)), S_{1}(x)\bigr)
        \end{align*}
        for $t\in\R^d$ and $x\in f^{-1}(E_{i})$, for each $i=1,2,\ldots, m_{a}$.
\end{defi}

The following lemma will be useful when we use ergodic theorems later.
Here, $\mu_{\mathrm{L}}$ is the Lebesgue measure on $\mathbb{T}^d$, which
is measure theoretically identified with $[0,1)^d$.
It is in the proof of this lemma where we use the fact that $S_{1}$ is invertible and ergodic.
\begin{lem}\label{lem_ergodicity_R}
        The transformation\/ $R$ is\/ $\mu_{\mathrm{L}}\times\nu$-preserving
        and ergodic.
\end{lem}
\begin{proof}
     This will be proved in the appendix.
\end{proof}

We apply the Furstenberg--Kesten theorem and the multiplicative
ergodic theorem by Oseledets to $\mathbb{T}^d\times Y$, $R$ and $C$. 
 In
order to invoke these theorem, we first have to confirm the following
lemma. 

\begin{lem}
      The two maps that send\/ $(z,x)\in\T^d\times Y$ to\/ $\log^{+}\|C(z,x)\|$ and to\/
      $\log^{+}\|C(z,x)^{-1}\|$, respectively,
       are both integrable with respect to\/ $\mu_{\mathrm{L}}\times\nu$ (the product measure of\/ 
       $\mu_{\mathrm{L}}$ and\/ $\nu$).
\end{lem}
\begin{proof}
       It suffices to prove that the maps that send $z\in\T^d$ to $\log\|C^{(i)}(z)^{\pm 1}\|$ are
       integrable for each $i$. 
       
       We take $i$ and fix it. Since $\|C^{(i)}(z)\|$ is bounded for
       $z\in\T^d$, we see that $\log^{+}\|C^{(i)}(z)\|$ is integrable. To
       prove the integrability of $\log^{+}\|C^{(i)}(z)^{-1}\|$, note
       that $C^{(i)}(z)^{-1}=\frac{1}{|\det
         C^{(i)}(z)|}C^{(i)}(z)^{\text{ad}}$, where $\text{ad}$
       denotes the adjugate matrix. It suffices to show that
       $\log|\det C^{(i)}(z)|$ is integrable, and this is proved by
       \cite[p.223, Lemma 2]{Schinzel}.
\end{proof}

In what follows, we set
\begin{align*}
       C_{n}(\omega)=C(R^{n-1}(\omega))\, C(R^{n-2}(\omega))\, \cdots\, C(\omega),
\end{align*}
for each $\omega\in\T^d\times Y$ and $n\in\Zpo$.
(It is customary to denote this by $C^{n}$, but to clearly distinguish this from $C^{(i)}$,
we prefer to use this notation.)

\begin{prop}\label{FK_and_Oseledets_thm}
      There are\/ $\chi_{+},\chi_{-}\in\R$ such that
      \begin{align*}
            \chi_{+}&=\lim_{n\rightarrow\infty}\frac{1}{n}\log\| C_n(\omega)\|,\\
            \chi_{-}&=\lim_{n\rightarrow\infty}\frac{1}{n}\log\| C_n(\omega)^{-1}\|^{-1},
      \end{align*}
      for almost all\/ $\omega\in\T^d\times Y$.
      
      We also have
      \begin{align}
             \chi_{+}+\chi_{-}=\lim_{n\rightarrow\infty}\frac{1}{n}\log|\det  C_n(\omega)|
             \label{eq_det_exponents}
      \end{align}
      for almost all\/ $\omega\in\T^d\times Y$.
      
      For each\/ $v\in\R^{2}\setminus\{0\}$, the limit
      \begin{align*}
            \lim_{n\rightarrow\infty}\frac{1}{n}\log\| C_n(\omega)v\|
      \end{align*}
      converges for almost all\/ $\omega$ and the limit is either\/
      $\chi_{+}$ or\/ $\chi_{-}$.
      
\end{prop}
\begin{proof}
      This is a direct consequence of the Furstenberg--Kesten theorem
      \cite[Theorem 3.12]{Viana} and the Oseledets theorem
      \cite[Theorem 4.1 and Section 4.3.3]{Viana}, if we modify the
      latter to the case of flags inside $\C^{2}$, not $\R^{2}$.
\end{proof}

We then analyze $\chi_{+}$ and $\chi_{-}$ to obtain the
estimate \eqref{eq_liminf}. These are related to logarithmic Mahler
measures for certain polynomials, as follows.
\textcolor{red}{(The logarithmic Mahler measure $\mathfrak{m}(p)$ for a complex-coefficient,
$d$-variable polynomial $p$ is defined via}
\begin{align*}
       \mathfrak{m}(p)=\int_{\mathbb{T}^d}\log|p(z_1,z_2,\ldots, z_d)|dz_1dz_2\cdots dz_d
\end{align*}
\textcolor{red}{where the integral is with respect to the (normalized) Lebesgue measure.
For Mahler measures, see for example \cite[p.224]{Schinzel}.)} First, we define
polynomials associated with each substitution
$\rho_{1},\rho_{2},\ldots, \rho_{m_{a}}$.
\begin{defi}
       For each $i=1,2,\ldots, m_{a}$ and $k,l=1,2$, we set
       \begin{align*}
              S^{(i)}_{k,l}=T^{(i)}_{k,1}\cap T^{(i)}_{l,2}.
       \end{align*}
       In other words, $S^{(i)}_{k,l}$ is the set of all places in
       $F_i=\phi_i([0,1)^d)\cap\mathbb{Z}^d$ where
       in the image of $\phi_i(T_1)$ there is a $T_k$ and in the image of
       $\phi_i(T_2)$ there is a $T_l$.
       
       Define polynomials $q^{(i)}_{k,l}$ (for $i=1,2,\ldots, m_{a}$ and $k,l=1,2$) as follows.
       First set
       \begin{align*}
              z^f=(z_1^{f_1},z_2^{f_2},\ldots, z_d^{f_d})
       \end{align*}
       for $z=(z_1,z_2,\ldots, z_d)\in\mathbb{T}^d$ and $f=(f_1,f_2,\ldots, f_d)\in\mathbb{Z}^d$.
       Next, set
       \begin{align*}
             q^{(i)}_{k,l}(z)=\sum_{f\in S^{(i)}_{k,l}}z^{f}
       \end{align*}
       for each $z\in\mathbb{T}^d$.
\end{defi}

\begin{ex}
      If $\rho_{i}$ is the Thue--Morse substitution,
      \textcolor{red}{if we write it symbolically,} it sends $1$ to $12$ and $2$ to $21$. 
      We have
      \begin{align*}
               S^{(i)}_{1,1}=\emptyset, \quad S^{(i)}_{1,2}=\{0\}, \quad
               S^{(i)}_{2,1}=\{1\}, \quad S^{(i)}_{2,2}=\emptyset,
      \end{align*}
      and the polynomials are
      \begin{align*}
              q^{(i)}_{1,1}(z)=0,\quad q^{(i)}_{1,2}(z)=1,\quad q^{(i)}_{2,1}(z)=z,\quad q^{(i)}_{2,2}(z)=0.
      \end{align*}
\end{ex}

\begin{rem}
    It follows that for each $i\in\{1,2,\ldots, m_{a}\}$, we have
    \begin{align*}
           C^{(i)}(z)=
           \begin{pmatrix}
                  q^{(i)}_{1,1}(z)+q^{(i)}_{1,2}(z), & q^{(i)}_{1,1}(z)+q^{(i)}_{2,1}(z)\\
                  q^{(i)}_{2,1}(z)+q^{(i)}_{2,2}(z), & q^{(i)}_{1,2}(z)+q^{(i)}_{2,2}(z)
           \end{pmatrix}.
    \end{align*}
\end{rem}

Then, we obtain the following results, Lemma
\ref{lem_eigenvector_for_C}, Lemma~\ref{lem_whatis_lambda+-} and Lemma
\ref{lem_hyouka_mahlermeasure}, which are modifications of results from
\cite{Manibo_binary, Manibo_thesis}.

\begin{lem}\label{lem_eigenvector_for_C}
        For each\/ $i$ and\/ $z\in\T^d$, we have
     \begin{align*}
       \det C^{(i)}(z)=
       \bigl(q^{(i)}_{1,2}(z)-q^{(i)}_{2,1}(z)\bigr)\sum_{f\in F_i}z^f.
    \end{align*}
    Moreover, for each\/ $i$ and\/ $z$, the vector\/ $(1,-1)^{\top}$
    is an eigenvector of\/ $C^{(i)}(z)$ with eigenvalue\/
    $q^{(i)}_{1,2}(z)-q^{(i)}_{2,1}(z)$.
\end{lem}

Note that the assumption (Setting \ref{setting_binary_const_length})
of non-singularity of the $B^{(i)}(t)$, and hence of the $C^{(i)}(z)$,
implies that, for all $i$, the polynomials
$q^{(i)}_{1,2}(z)-q^{(i)}_{2,1}(z)$ are not zero.

To analyze $\chi_{+}$ and $\chi_{-}$, we first prove the following.
\begin{lem}\label{lem_whatis_lambda+-}
 We have\/ $\chi_{+}=0$ and\/ 
             $\chi_{-}=-\sum_{i=1}^{m_{a}}\mu(E_{i})\,\mathfrak{m}(q^{(i)}_{1,2}-q^{(i)}_{2,1})$.
\end{lem}

\begin{proof}
          By equation \eqref{eq_det_exponents}, Lemma \ref{lem_ergodicity_R} and
          Birkhoff's ergodic theorem, for almost all $\omega\in\T^d\times Y$, we have
          \begin{align}
                   \chi_{+}+\chi_{-}
                   &=\int\log|\det C|\, d\mu_{\mathrm{L}}\times\nu\nonumber\\
                   &=-\sum_{i=1}^{m_{a}}\mu(E_{i})\, \mathfrak{m}(q^{(i)}_{1,2}-q^{(i)}_{2,1})\label{eq_for_lambda+-}
          \end{align}
          Note that we used the fact that the logarithmic Mahler measure for \textcolor{red}{
          $\sum_{f\in F_i}z^f$ vanishes by Jensen's formula \cite[p.207]{Ahlfors}.}
          
          On the other hand, by Lemma \ref{lem_eigenvector_for_C}, 
          setting 
          \begin{align*}
                v=\begin{pmatrix}1\\-1\end{pmatrix}
          \end{align*}
          we have
          \begin{align*}
                 \frac{1}{n}\log\big\| C_n(z,x)v\|
              \rightarrow-\sum_{i=1}^{m_{a}}\mu(E_{i})\,\mathfrak{m}(q^{(i)}_{1,2}-q^{(i)}_{2,1})
          \end{align*}
           again by Birkhoff's
          ergodic theorem.  By Proposition \ref{FK_and_Oseledets_thm},
          the last value is either $\chi_{+}$ or $\chi_{-}$.  Since
          $\chi_{+}\geqq\chi_{-}$ and
          $\mathfrak{m}(q^{(i)}_{1,2}-q^{(i)}_{2,1})\geqq 0$ \textcolor{red}{by
          \cite[p.228, Lemma 3]{Schinzel}, }by using equation
          \eqref{eq_for_lambda+-}, we have the claim.
\end{proof}

\begin{lem}\label{lem_hyouka_mahlermeasure}
        For each\/ $i=1,2,\ldots, m_{a}$, we have
        \begin{align*}
                   \mathfrak{m}(q^{(i)}_{1,2}-q^{(i)}_{2,1})<\log\sqrt{\det\phi_i}.
        \end{align*}
\end{lem}
\begin{proof}
      By Jensen's inequality and H\"older's inequality, for each $i$,
     \begin{align*}
       \exp\mathfrak{m}(q^{(i)}_{1,2}-q^{(i)}_{2,1})
             \leqq \sqrt{\det\phi_i}.
     \end{align*}
        \textcolor{red}{If the polynomial $q^{(i)}_{1,2}-q^{(i)}_{2,1}$ is not a monomial, its absolute value is
          not a constant function and so the Jensen's inequality is strict.
        If the polynomial
        $q^{(i)}_{1,2}-q^{(i)}_{2,1}$ is a monomial we have also a strict inequality.}
\end{proof}

We now prove the main result of this section. We consider S-adic tilings belonging to
a directive sequence $i_{1},i_{2},\ldots$ that is obtained by a coding of $S_0$ starting at
some $x\in X$. In other words, $i_{1},i_{2},\ldots$ is the sequence of $\{1,2,\ldots, m_{a}\}$
that satisfies
\begin{align*}
         S_0^{n-1}(x)\in E_{i_{n}}
\end{align*}
for each $n=1,2,\ldots$.

\begin{thm}\label{thm_absence_for_binary}
\textcolor{red}{
      Suppose that the substitution matrix $A_i$ for each $\rho_i$ has only strictly positive entries, or more generally,}
             suppose that
       \begin{enumerate}
       \item there are\/ $i_{0,1},i_{0,2},\ldots, i_{0,n_{0}}\in\{1,2,\ldots, m_{a}\}$ such that
                the substitution matrix
                \begin{align*}
                       A_{i_{0,1}}A_{i_{0,2}}\cdots A_{i_{0,n_{0}}}
                \end{align*}
                for\/ $\rho_{i_{0,1}}\circ\rho_{i_{0,2}}\cdots\circ\rho_{i_{0,n_{0}}}$
                has only strictly positive entries, and that
        \item 
        \begin{align}
         \mu(\bigcap_{j=1}^{n_{0}}S_{0}^{-(j-1)}E_{i_{0,j}})>0\label{positive_measure_for_frequency}.
        \end{align}
       \end{enumerate}
      Then, there is a set\/ $E\in\mathcal{B}$ of full measure such
      that, if we take\/ $x\in E$ and define a directive sequence\/
      $i_{1},i_{2}\cdots$ as the coding of\/ $S_0$ starting at\/ $x$
      with respect to the decomposition\/ $X=\bigcup E_{i}$, and if\/
      $\mathcal{T}$ is an S-adic tiling which belongs to this
      directive sequence for the family $\{\rho_1,\rho_2,\ldots,\rho_{m_a}\}$, then\/ $\mathcal{T}$ has zero absolutely
      continuous diffraction spectrum.
\end{thm}

\begin{proof}
       By Birkhoff's ergodic theorem, for almost all $x\in X$ we have
       \begin{align*}
               \lim_{n\rightarrow\infty}\frac{1}{2n}\log(\det\phi_{i_{1}}\det\phi_{i_{2}}\cdots \det\phi_{i_{n}})=\frac{1}{2}
               \sum_{i=1}^{m_{a}}\mu(E_{i})\log \det\phi_{i}.
       \end{align*}
       (Here, $i_{1},i_{2},\ldots$ are dependent on $x$.)
       By Proposition \ref{FK_and_Oseledets_thm} and
       Lemma \ref{lem_whatis_lambda+-}, for almost all $y\in Y$  and $t\in\R^d$, 
       if $i_{1},i_{2},\ldots$ are codings of $S_0$ starting at $f(y)$, then
       we have
       \begin{align}
              &\lim_{n\rightarrow\infty}\frac{1}{n}\log\big\|B^{(i_{1})}(t)B^{(i_{2})}(\phi_{i_{1}}(t))\cdots
              B^{(i_{n})}(\phi_{i_{n-1}}\circ\phi_{i_{n-2}}\circ\cdots \circ\phi_{i_{1}}(t))\big\|\nonumber\\
              &=\lim_{n\rightarrow\infty}\frac{1}{n}\log\big\|C^{(i_{1})}(\pi(t))C^{(i_{2})}(\pi(\phi_{i_{1}}(t)))\cdots
                   C^{(i_{n})}(\pi(\phi_{i_{n-1}}\phi_{i_{n-2}}\cdots \phi_{i_{1}}(t)))\big\|\nonumber\\
             &=\lim_{n\rightarrow\infty}\frac{1}{n}\log\big\|C(\pi(t),y)^{-1}C(R(\pi(t),y))^{-1}\cdots
                   C(R^{n-1}(\pi(t),y))^{-1}\big\|\nonumber\\
                  & =\sum_{i=1}^{m_{a}}\mu(E_{i})\,\mathfrak{m}(q^{(i)}_{1,2}-q^{(i)}_{2,1}).\label{eq_lyap_expo}
       \end{align}
       By Lemma \ref{lem_hyouka_mahlermeasure}, there is $\e>0$ such that the number
       \eqref{eq_lyap_expo} is equal to
       \begin{align*}
                \frac{1}{2}\sum_{i=1}^{m_{a}}(\mu(E_{i})\log(\det\phi_{i}))-\e.
       \end{align*}
       These computations show that the inequality \eqref{eq_liminf} is satisfied.
       (Note that $\phi_i^*=\phi_i$.)
       
       In Setting \ref{setting_relation_exponents_absconti-spec},
       under which Theorem \ref{thm_liminf_implies_zero_acpart} is
       proved, we assumed the patch frequencies for each
       $\mathcal{T}^{(k)}$ to be convergent. In order to use Theorem
       \ref{thm_liminf_implies_zero_acpart}, we have to prove this
       assumption for the present case.  By
       \eqref{positive_measure_for_frequency}, ergodicity of $S_0$ and
       Poincar\'e recurrence, for almost all $x\in X$, there are
       infinitely many $n\in\Zpo$ such that
       $S_0^{n}(x)\in\bigcap_{j=1}^{n_{0}}S_0^{-(j-1)}E_{i_{0,j}}$.  (We
       used \cite[Theorem 1.5 (iii)]{Walters}.) This means that there
       are infinitely many $n$ such that
      \begin{align*}
              i_{n}i_{n+1}\cdots i_{n+n_{0}-1}=i_{0,1}i_{0,2}\cdots i_{0,n_{0}}.
      \end{align*}
      (These are equal as words.)
      The patch frequencies converge by Theorem \ref{thm_S-adic_patch-freq}.
      
      The statement, the absence of absolutely continuous spectrum, is proved by using
      Theorem \ref{thm_liminf_implies_zero_acpart}.
\end{proof}

\begin{ex}\label{example_absense_acpart_binary}
      Let $\rho_{1}$ be the period-doubling substitution and
      $\rho_{2}$ be the Thue--Morse substitution. Let
      $X=\{1,2\}^{\mathbb{N}}$ and $(X,\mathcal{B},\mu,S_0)$ be the
      Bernoulli shift for some probability vector $(p_{1},p_{2})$ with
      $p_{1}p_{2}\neq 0$.  \textcolor{red}{Then,} for almost all $x\in X$, the S-adic
      tilings belonging to the sequence $x_{1},x_{2},\ldots$ have zero
      absolutely continuous spectrum.  We can replace
      $(X,\mathcal{B},\mu,S_0)$ with an arbitrary surjective ergodic
      measure-preserving system on a standard probability space and
      get the same conclusion.  For example, we take some Sturmian sequences
      as a directive sequence and obtain the same conclusion, by replacing
      $(X,\mathcal{B},\mu,S_{0})$ with an irrational rotation on $\mathbb{T}$.
      \textcolor{red}{The same conclusion holds when we replace $\rho_1$ and $\rho_2$ with arbitrary two binary substitution rules
      such that the supports of the prototiles are all $[0,1]$ and the substitution matrices have only strictly positive entries. 
      The expansion factors can be different.}
      
     \textcolor{red}{We can also apply Theorem \ref{thm_absence_for_binary} for higher-dimensional  block substitutions.
      We can replace the above $\rho_1$ and $\rho_2$ with arbitrary two block substitutions in $\R^d$
      with substitution matrices with only strictly positive entries.
      For example, for $d=2$, if $\rho_1,\rho_2$ are ones that were depicted in Figure \ref{fig_block_substi1} and \ref{fig_block_substi2},
      we have the same conclusion as in the first paragraph in this example.}
      
      \textcolor{red}{Moreover, the number of substitutions does not have to be 2.
      Let $\rho_1,\rho_2,\ldots, \rho_{m_a}$ be a finite family of block substitutions in $\R^d$.
      Take a space $\{1,2,\ldots, m_a\}^{\mathbb{N}}$ and endow it the product probability measure for
      the probability measure on $\{1,2,\ldots, m_a\}$ given by a probability vector $(p_1,p_2,\ldots, p_{m_a})$ with
      $p_1p_2\cdots p_{m_a}\neq 0$. 
      We can also relax the assumption on the substitution matrices and the substitution matrices 
      may have 0 as entries, but we assume there are $i_{0,1},i_{0,2},\ldots, i_{0,n_0}\in\{1,2,\ldots, m_a\}$
      such that the product $A_{i_{0,1}}A_{i_{0,2}}\cdots A_{i_{0,n_0}}$ has only strictly positive entries.
      Then, for almost all $x\in\{1,2,\ldots, m_a\}^{\mathbb{N}}$, the S-adic tilings
      constructed by $\rho_1,\rho_2,\ldots, \rho_{m_a}$ belonging to the directive sequence $x$ has
      zero absolutely continuous diffraction spectrum.
      We can also take other subshifts of $\{1,2,\ldots, m_a\}^{\mathbb{N}}$ and obtain the same conclusion.
      In this case, we assume the existence of $i_{0,1}i_{0,2}\cdots i_{0,n_0}$ in the language with the same condition.
      For example, let $X$ be the hull of a $m_a$-letter primitive symbolic substitution. Endow $X$ with a shift-invariant probability 
      measure defined by frequencies. Then, with respect to that measure, for almost all $x\in X$, the S-adic tilings
      belonging to $x$ has zero absolutely continuous diffraction spectrum.
      }
%
\end{ex}

\begin{rem}
      In Theorem \ref{thm_absence_for_binary}, the author could not
      prove the absence of absolutely continuous part for an
      \emph{arbitrary} directive sequence, since this theorem is an
      ``almost sure'' result.  \textcolor{red}{For example, when $m_a=2$, we do not know what happens if we take
      the Fibonacci sequence as a directive sequence.} It would be interesting to decide
      whether this theorem holds for arbitrary directive sequences, or
      whether for some directive sequences there may be non-vanishing
      absolutely continuous components in the diffraction spectrum.
\end{rem}

\section*{Appendix: the ergodicity of $R$ in Section \ref{subsection_absence_acpart}}
Here, we prove the ergodicity of the transformation $R$ in Section
\ref{subsection_absence_acpart}.  In particular, let
$(Y,\mathcal{C},\nu)$ be a probability space and
$Y=\bigcup_{i=1}^{m_{a}}E_{i}'$ be a decomposition of $Y$ into
pairwise disjoint $E_{i}'=f^{-1}(E_{i})\in\mathcal{C}$, for $i=1,2,\ldots, m_{a}$. Let
$S=S_{1}\colon Y\rightarrow Y$ be an invertible, measure-preserving and ergodic
transformation.  
For each $i$, take integers $n^{(i)}_1, n^{(i)}_2, \ldots, n^{(i)}_d$ greater than $1$ and let $\phi_i$ be the linear map
defined by the diagonal matrix with diagonal entries $n^{(i)}_1, n^{(i)}_2,\ldots , n^{(i)}_d$.
Define a transformation $R\colon\mathbb{T}^d\times
Y\rightarrow\mathbb{T}^d\times Y$ via
\begin{align*}
       R(\pi(t),x)=(\pi(\phi_i(t)), S(x)),
\end{align*}
for each $t\in\R^d$, $x\in E_{i}'$ and $i=1,2,\ldots, m_{a}$. We aim to prove the following result,
where\/ $\mu_{\mathrm{L}}$ is the Lebesgue measure on\/ $[0,1)^d$.
\begin{lem}\label{lem_R_ergodic}
        The transformation\/ $R$ is measure preserving and ergodic with respect to
        the product measure $\mu_{\mathrm{L}}\times\nu$.
\end{lem}

We divide the proof into several lemmas. First, we define a notion as follows.
\begin{defi}
    For each $n=(n_1,n_2,\ldots, n_d) \in\Zpo^d$, define a map $T_{n}$ on $[0,1)^d$ via
    \begin{align*}
           T_{n}(t_1,t_2,\ldots, t_d)=(n_1t_1-\lfloor n_1t_1\rfloor, n_2t_2-\lfloor n_2t_2\rfloor,,\ldots , n_dt_d-\lfloor n_dt_d\rfloor)
    \end{align*}
    where, for each $s\in\R$,  $\lfloor s \rfloor$ denotes the largest integer smaller than
    or equal to $s$.
    
    For each $n=(n_1,n_2,\ldots, n_d) \in\Zpo^d$, set
    \begin{align*}
            \mathscr{I}_{n}=\left\{\left [\frac{k_1}{n_1} ,\frac{k_1+1}{n_1}\right )\times
                                               \left [\frac{k_2}{n_2} ,\frac{k_2+1}{n_2}\right )\times
                                               \cdots\times
                                               \left [\frac{k_d}{n_d} ,\frac{k_d+1}{n_d}\right )
                                               \mathrel{}\middle|\mathrel{}
                                                k_j=0,1,\ldots, n_j-1
                                     \right\}.
    \end{align*}
\end{defi}

Note that for each $I\in\mathscr{I}_{n}$, the restriction
$\maprestriction{T_{n}}{I}\colon I\rightarrow [0,1)$ is bijective and
  expands the Lebesgue measure by $n_1n_2\cdots n_d$.

The set $\Zpo^d$ is a directed set. Any real-valued map with domain $\Zpo^d$ can be regarded as a net
and the limit of such a map makes sense. In other words,  a map $f\colon\Zpo\rightarrow\R$ converges to
an $\alpha\in\R$ if for any $\e>0$, there are $(n_1,n_2,\ldots, n_d)\in\Zpo^d$ such that
$m=(m_1,m_2,\ldots, m_d)\in\Zpo^d$ and $m_1>n_1, m_2>n_2,\ldots , m_d>n_d$ imply
$|f(m)-\alpha|<\e$.

The proof of the following lemma is straightforward.
\begin{lem}\label{lem1_for_ergodicity}
      If\/ $I_{0}\subset [0,1)^d$ is a product of intervals, then we have
      \begin{align*}
             \lim_{n=(n_1,n_2,\ldots, n_d)\rightarrow\infty}\frac{1}{n_1n_2\cdots n_d}\card\{I\in\mathscr{I}_{n}\mid I\subset I_{0}\}=\mu_{\mathrm{L}}(I_{0}).
      \end{align*}
\end{lem}

By using Lemma \ref{lem1_for_ergodicity}, we prove the following lemma.
\begin{lem}\label{lem_T_n_convergent}
       For all Borel subsets\/ $E,F$ of\/ $[0,1)$, we have
       \begin{align*}
            \lim_{n\rightarrow\infty} \mu_{\mathrm{L}}(T_{n}^{-1}(E)\cap F)=\mu_{\mathrm{L}}(E)\,\mu_{\mathrm{L}}(F).
       \end{align*}
\end{lem}
\begin{proof}
        It suffices to prove the claim for the case where $E$ and $F$
        are products of intervals, since any general $E$ and $F$ are approximated
        by finite disjoint unions of products of intervals.  Assume that $E$ and $F$
        are products of  intervals.  For each $n\in\Zpo^d$ and $I\in\mathscr{I}_{n}$, we
        have either (1) $I\subset F$, (2) $I\cap F\neq\emptyset$ and
        $I\not\subset F$, or (3) $I\cap F=\emptyset$.  In case (1),
        we have
        \begin{align*}
                \mu_{\mathrm{L}}\bigl(\maprestriction{T_{n}}{I}^{-1}(E)\cap F\bigr)=\frac{1}{n}\mu_{\mathrm{L}}(E).
        \end{align*}
        For each $n$, the number of $I$ in $\mathscr{I}_{n}$ with case (2) is at most
        \begin{align*}
               2\sum_{j=1}^dn_1n_2\cdots n_{j-1}n_{j+1}\cdots n_d.
        \end{align*}
         For each such $I$, the measure
        $\mu_{\mathrm{L}}(\maprestriction{T_{n}}{I}^{-1}(E)\cap F)$ is at most $\frac{1}{n_1n_2\cdots n_d}$.
        Finally for $I$ with case (3), we have
        $\maprestriction{T_{n}}{I}^{-1}(E)\cap F=\emptyset$.
        
        By using these observations and Lemma \ref{lem1_for_ergodicity}, we see that
        \begin{align*}
              \lim_{n\rightarrow\infty}\mu_{\mathrm{L}}(T_{n}^{-1}(E)\cap F)&=
              \lim\sum_{I\in\mathscr{I}_{n}, I\subset F}\mu_{\mathrm{L}}\bigl(\maprestriction{T_{n}}{I}^{-1}(E)\cap F\bigr)\\
              &=\mu_{\mathrm{L}}(E)\,\mu_{\mathrm{L}}(F).\qedhere
        \end{align*}
 \end{proof}

\begin{defi}
      For each $n\in\Zpo$, define $\mathfrak{I}(n)=\{1,2,\ldots, m_{a}\}^{n}$.
      For each element $\mathfrak{i}=(i_{1},i_{2},\ldots, i_{n})\in\mathfrak{I}(n)$, set
      \begin{align*}
           E'_{\mathfrak{i}}=S^{n}(E'_{i_{1}})\cap S^{n-1}(E'_{i_{2}})\cap\cdots\cap S(E'_{i_{n}}).
      \end{align*}
      We also define a map $T_{\mathfrak{i}}\colon\mathbb{T}^d\rightarrow\mathbb{T}^d$
      for $\mathfrak{i}=(i_{1},i_{2},\ldots, i_{n})$, via
      \begin{align*}
        T_{\mathfrak{i}}(\pi(t))=\pi(\phi_{i_{n}}\circ\phi_{i_{n-1}}\circ\cdots \circ\phi_{i_{1}}(t))
      \end{align*}
      for each $t\in\R^d$. 
\end{defi}

\begin{lem}\label{lem2_for_ergodicity}
       Let\/ $n\in\Zpo$ and\/ $\mathfrak{i}\in\mathfrak{I}(n)$.
       For a Borel subset\/ $F\subset\mathbb{T}^d$ and\/ $E\in\mathcal{C}$ with\/
       $E\subset E'_{\mathfrak{i}}$,
       we have
       \begin{align*}
               R^{-n}(F\times E)=(T_{\mathfrak{i}}\times S^{n})^{-1}(F\times E).
       \end{align*}
\end{lem}
\begin{proof}
         For a $(\pi(t),x)\in\mathbb{T}^d\times Y$, if we have $R^{n}(\pi(t),x)\in F\times E$, then
         $S^{n}(x)\in E'_{\mathfrak{i}}$, and
         \begin{align*}
               x\in E'_{i_{1}},\quad S(x)\in E'_{i_{2}},\quad \ldots,\quad S^{n-1}(x)\in E'_{i_{n}},
         \end{align*}
          where
         $\mathfrak{i}=(i_{1},i_{2},\ldots, i_{n}).$
         This implies that 
         \begin{align*}
           R(\pi(t),x)=(\pi(\phi_{i_{1}}(t)),S(x)), \quad
           R^{2}(\pi(t),x)=(\pi(\phi_{i_{2}}\circ\phi_{i_{1}}(t)),S^{2}(x)),
         \ldots,
         \end{align*}
        and
         \begin{align*}
                 R^{n}(\pi(t),x)&=(\pi(\phi_{i_{n}}\circ \phi_{i_{n-1}}\circ\cdots \circ\phi_{i_{1}}(t)),S^{n}(x))\\
                 &=(T_{\mathfrak{i}}(\pi(t)), S^{n}(x)).
         \end{align*}
         This proves that $R^{-n}(F\times E)\subset(T_{\mathfrak{i}}\times S^{n})^{-1}(F\times E)$.
         The reverse inclusion is proved in a similar way.
\end{proof}

\begin{lem}
       The transformation\/ $R$ is measure preserving with respect to the product measure\/
       $\mu_{\mathrm{L}}\times\nu$.
\end{lem}
\begin{proof}
         For $E\in\mathcal{C}$ and a Borel $F\subset \mathbb{T}^d$, 
         by Lemma \ref{lem2_for_ergodicity} we have
         \begin{align*}
             \mu_{\mathrm{L}}\times\nu(R^{-1}(F\times E))&
              =\sum_{i=1}^{m_{a}}\mu_{\mathrm{L}}\times\nu\bigl(R^{-1}(F\times (E\cap E'_{i}))\bigr)\\
              &=\sum_{i=1}^{m_{a}}\mu_{\mathrm{L}}\times\nu\bigl((T_{i}\times S)^{-1}(F\times (E\cap E'_{i}))\bigr)\\
              &=\mu_{\mathrm{L}}\times\nu(F\times E).\qedhere
         \end{align*}
\end{proof}

\begin{lem}
      $R$ is ergodic with respect to\/ $\mu_{\mathrm{L}}\times\nu$.
\end{lem}
\begin{proof}
      For Borel sets $F_{1},F_{2}\subset\mathbb{T}^d$, $G_{1},G_{2}\in\mathcal{C}$ and $n\in\Zpo$,
      we have
      \begin{align*}
              \mu_{\mathrm{L}}\times\nu\bigl(R^{-n}(F_{1}\times G_{1})\cap (F_{2}\times G_{2})\bigr)
              &=\sum_{\mathfrak{i}\in\mathfrak{I}(n)}\mu_{\mathrm{L}}\times\nu\bigl(R^{-n}(F_{1}\times (G_{1}\cap E'_{\mathfrak{i}}))
               \cap (F_{2}\times G_{2})\bigr)\\
               &=\sum_{\mathfrak{i}\in\mathfrak{I}(n)}\mu_{\mathrm{L}}\bigl(T_{\mathfrak{i}}^{-1}(F_{1})\cap F_{2}\bigr)\,
               \nu\bigl(S^{-n}(G_{1}\cap E'_{\mathfrak{i}})\cap G_{2}\bigr),
      \end{align*}
      where we used Lemma \ref{lem2_for_ergodicity} for the second equality.
      
      By Lemma \ref{lem_T_n_convergent}, for arbitrary $\e>0$, there is $n_{0}>0$ such that,
      if $n\geqq n_{0}$ and $\mathfrak{i}\in\mathfrak{I}(n)$, the number
      \begin{align*}
           \e_{\mathfrak{i}}=\mu_{\mathrm{L}}(T_{\mathfrak{i}}^{-1}(F_{1})\cap F_{2})-\mu_{\mathrm{L}}(F_{1})\mu_{\mathrm{L}}(F_{2})
      \end{align*}
      has an estimate
      \begin{align*}
              |\e_{\mathfrak{i}}|<\e.
      \end{align*}
      We therefore have that
      \begin{align*}
                          \mu_{\mathrm{L}}\times\nu\bigl(R^{-n}(F_{1}\times G_{1})\cap (F_{2}\times G_{2})\bigr)
                          =\sum_{\mathfrak{i}\in\mathfrak{I}(n)}\mu_{\mathrm{L}}(F_{1})\,\mu_{\mathrm{L}}(F_{2})\,
                          \nu\bigl(S^{-n}(G_{1}\cap E'_{\mathfrak{i}})\cap G_{2}\bigr)+
                          \e_{n},
      \end{align*}
      where
      \begin{align*}
        \e_{n}=\sum_{\mathfrak{i}\in\mathfrak{I}(n)}\e_{\mathfrak{i}}\,
        \nu\bigl(S^{-n}(G_{1}\cap E'_{\mathfrak{i}})\cap G_{2}\bigr).
      \end{align*}
      Since the $E'_{\mathfrak{i}}$ (with $\mathfrak{i}\in\mathfrak{I}(n)$) give a partition of $X$, 
      we firstly see that $|\e_{n}|<\e$, and secondly conclude that
      \begin{align*}
              \sum_{\mathfrak{i}\in\mathfrak{I}(n)}
              \nu\bigl(S^{-n}(G_{1}\cap E'_{\mathfrak{i}})\cap G_{2}\bigr)
              =\nu\bigl(S^{-n}(G_{1})\cap G_{2}\bigr).
      \end{align*}
      
      For each $N>0$, we have
      \begin{align*}
            \frac{1}{N}\sum_{n=0}^{N-1}\mu_{\mathrm{L}}\times\nu\bigl(R^{-n}(F_{1}\times G_{1})\cap (F_{2}\times G_{2})\bigr)
            =\mu_{\mathrm{L}}(F_{1})\,\mu_{\mathrm{L}}(F_{2})\,\frac{1}{N}\sum_{n=0}^{N-1}
            \bigl(\nu(S^{-n}(G_{1})\cap G_{2})+\e_{n}\bigr).
      \end{align*}
     By ergodicity of $S$, as $N\rightarrow\infty$ this converges to
     \begin{align*}
       \mu_{\mathrm{L}}(F_{1})\,\mu_{\mathrm{L}}(F_{2})\,\nu(G_{1})\,\nu(G_{2})
       =\mu_{\mathrm{L}}\times\nu(F_{1}\times G_{1})\,\mu_{\mathrm{L}}\times\nu(F_{2}\times G_{2}),
     \end{align*}
     up to an error term of absolute value less than $\e$.
     Since $\e$ was arbitrary, we see that
     \begin{align*}
       \lim_{N\rightarrow\infty} \frac{1}{N}\sum_{n=0}^{N-1}
       \mu_{\mathrm{L}}\times\nu\bigl(R^{-n}(F_{1}\times G_{1})\cap (F_{2}\times G_{2})\bigr)
       =\mu_{\mathrm{L}}\times\nu(F_{1}\times G_{1})\,\mu_{\mathrm{L}}\times\nu(F_{2}\times G_{2}).
     \end{align*}
     Since $F_{1},F_{2},G_{1}$ and $G_{2}$ are arbitrary, this proves the ergodicity of $R$.
\end{proof}

\section*{Acknowledgment }
The author was supported by EPSRC grant EP/S010335/1.
The author thanks Uwe Grimm for fruitful discussions.
The author also thanks Hitoshi Nakada for a discussion on natural extensions and
Neil Ma\~nibo for valuable comments on the draft.
\textcolor{red}{The author thanks referees for valuable comments.}

\end{document}